\documentclass[12pt]{article}

\usepackage{amsmath,amssymb, amsthm,url,tikz,enumerate}

\usetikzlibrary{intersections,arrows,patterns,shapes}

\newcommand{\cA}{\mathcal{A}}

\newcommand{\cC}{\mathcal{C}}
\newcommand{\cD}{\mathcal{D}}
\newcommand{\cF}{\mathcal{F}}
\newcommand{\cG}{\mathcal{G}}
\newcommand{\cI}{\mathcal{I}}
\newcommand{\cR}{\mathcal{R}}
\newcommand{\cS}{\mathcal{S}}
\newcommand{\cT}{\mathcal{T}}

\newcommand{\bZ}{\mathbb{Z}}

\newtheorem{theorem}{Theorem}[section]

\newtheorem{proposition}{Proposition}[section]
\newtheorem{lemma}{Lemma}[section]
\newtheorem{observation}{Observation}[section]
\newtheorem*{question}{Question}

\theoremstyle{definition}

\begin{document}
\title{Growth rates of groups associated with face 2-coloured triangulations and directed Eulerian digraphs on the sphere}
\author{Thomas A. McCourt\thanks{Department of Mathematics and Statistics, Plymouth University, Drake Circus, Plymouth PL4 8AA.}}
\date{\small Keywords: 
	Face 2-coloured spherical triangulation; 
	directed Eulerian spherical digraph; 
    canonical group;
    abelian sand-pile group;
    latin bitrade.\\
\small Mathematics Subject Classification: 05C10, 05B15, 05C20, 05C25.
}

\maketitle

\begin{abstract}
Let $\cG$ be a properly face 2-coloured (say black and white) piecewise-linear triangulation of the sphere with vertex set $V$. Consider the abelian group $\cA_W$ generated by the set $V$, with relations $r+c+s=0$ for all white triangles with  vertices $r$, $c$ and $s$. The group $\cA_B$ can be defined  similarly, using black triangles. These groups are related in the following manner $\cA_W\cong\cA_B\cong\bZ\oplus\bZ\oplus\cC$ where $\cC$ is a finite abelian group.

The finite torsion subgroup $\cC$ is referred to as the canonical group of the triangulation. Let $m_t$ be the maximal order of $\cC$ over all properly face two-coloured spherical triangulations with $t$ triangles of each colour. By relating such a triangulation to certain directed Eulerian spherical embeddings of digraphs whose abelian sand-pile groups are isomorphic to the triangulation's canonical group we provide improved upper and lower bounds for $\lim \sup_{t\rightarrow\infty}(m_t)^{1/t}$.

  \bigskip\noindent \textbf{Keywords:} Face 2-coloured spherical triangulation; directed Eulerian spherical digraph; canonical group; abelian sand-pile group; latin bitrade.
\end{abstract}

\section{Introduction}
\label{sec:intro}
Let $G$ be a graph. We will denote the vertex set of $G$ by $V(G)$ and the edge set of $G$ by $E(G)$.
Suppose that there exists a face 2-coloured, black and white say, triangulation of the sphere, i.e. a \textit{spherical triangulation}, $\cG$ which has $G$ as its underlying graph. Denote the set of white faces by $W$ and the set of black faces by $B$.  As the faces are properly face 2-coloured, $G$ is Eulerian and, by a well known result of Heawood \cite{Hea}, regardless of whether or not $G$ is simple, $G$ has a proper vertex 3-colouring. If $G$ is simple, then the rotation at every vertex is a cycle, i.e. the triangulation is \textit{piecewise-linear}. See Figure \ref{fig:a_triangulation} for an illustration of a face 2-coloured spherical triangulation where the underlying graph is simple.
\begin{figure}[!hb]
\begin{center}
\begin{tikzpicture}[fill=gray!50, scale=1,vertex/.style={circle,inner sep=2,fill=black,draw}]

\coordinate (s0) at (8,5);
\coordinate (s1) at (1,1);
\coordinate (s2) at (1,5);
\coordinate (s3) at (8,1);
\coordinate (s4) at (3,3);

\coordinate (c0) at (4.5,5);
\coordinate (c1) at (4.5,1);
\coordinate (c2) at (1,3);
\coordinate (c3) at (8,3);
\coordinate (c4) at (3,2);

\coordinate (r1) at (4.5,3.5);
\coordinate (r2) at (6.5,4);
\coordinate (r3) at (2,2);

\filldraw (s1) to [out=0, in=270] (r1) to (c1) to (s1);
\filldraw (c2) to [out=20, in=170] (r1) to (s2) to (c2);

\filldraw (r1) -- (c0) -- (s3) -- cycle;
\filldraw (r2) -- (c3) -- (s3) -- cycle;
\filldraw (r2) -- (c0) -- (s0) -- cycle;
\filldraw (r3) -- (c4) -- (s1) -- cycle;
\filldraw (r3) -- (c2) -- (s4) -- cycle;
\filldraw (r1) -- (c4) -- (s4) -- cycle;

\filldraw[color=gray!50] (0,0) -- (s1) -- (c2) -- (0,3) -- cycle;
\draw (0,0) -- (s1);
\draw (0,3) -- (c2);
\draw (s1) -- (c2);

\filldraw[color=gray!50] (4.5,0) -- (c1) -- (s3) -- (9,0) -- cycle;
\draw (4.5,0) -- (c1);
\draw (9,0) -- (s3);
\draw (c1) -- (s3);

\filldraw[color=gray!50] (9,3) -- (c3) -- (s0) -- (9,6) -- cycle;
\draw (9,3) -- (c3);
\draw (9,6) -- (s0);
\draw (s0) -- (c3);

\filldraw[color=gray!50] (4.5,6) -- (c0) -- (s2) -- (0,6) -- cycle;
\draw (4.5,6) -- (c0);
\draw (0,6) -- (s2);
\draw (s2) -- (c0);

\node at (s0) [vertex, label=north:$s_0$]{};
\node at (s1) [vertex, label=south:$s_1$]{};
\node at (s2) [vertex, label=west:$s_2$]{};
\node at (s3) [vertex, label=east:$s_3$]{};
\node at (s4) [vertex, label=north:$s_4$]{};]{};

\node at (c0) [vertex, label=north east:$c_0$]{};
\node at (c1) [vertex, label=south west:$c_1$]{};
\node at (c2) [vertex, label=north west:$c_2$]{};
\node at (c3) [vertex, label=south east:$c_3$]{};
\node at (c4) [vertex, label=east:$c_4$]{};

\node at (r1) [vertex]{};
\node at (4.35,3) [label=east:$r_1$]{};
\node at (r2) [vertex, label=west:$r_2$]{};
\node at (r3) [vertex, label=west:$r_3$]{};

\end{tikzpicture}
\end{center}

\caption{A face 2-coloured spherical triangulation. A vertex, $r_0$, has been placed at infinity.}
\label{fig:a_triangulation}
\end{figure}
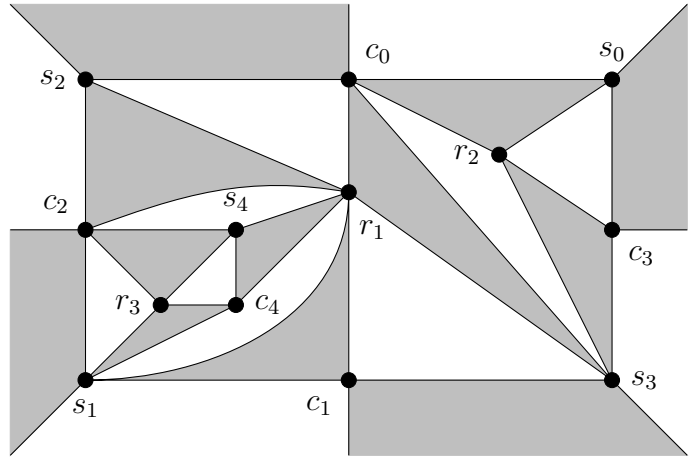

Define $\cA_W$ to be the abelian group with generating set $V(G)$, subject to the relations $\{r+c+s=0:r,c,s\text{ are the vertices of a white face of }\cG\}$. Define $\cA_B$ similarly but using the black faces. In \cite{SimonMe} Blackburn and the current author proved that $$\cA_W\cong\cA_B\cong\mathbb{Z}\oplus\mathbb{Z}\oplus\cC$$
where $\cC$ is a finite abelian group.  
In the same paper the question of the growth rate of the maximal order of $\cC$, in the terminology established in \cite{GrubWan} the \textit{canonical group} of the face 2-coloured spherical triangulation, was raised. More precisely:
\begin{question}[Blackburn \& McCourt, \cite{SimonMe}]
Let $m_t$ be the maximal order of the canonical group over all properly face 2-coloured spherical triangulations whose underlying graphs are simple and have $t$ faces of each colour. What is the value of $\limsup_{t\rightarrow\infty}\,(m_t)^{1/t}?$
\end{question}
In \cite{SimonMe} a lower bound of 1.201 was obtained. Earlier work of  Cavenagh and Wanless \cite{NickIan} provided an upper bound of $6^{1/3}<1.818$ and of Dr\'apal and Kepka \cite{DraKep} of $e^{1/e}<1.445$. More recently Grubman and Wanless \cite{GrubWan} improved the lower bound to $5123^{1/30}>1.329$. In Section \ref{sec:upper} we will provide an improved upper bound of $6^{1/5}<1.431$ and in Section \ref{sec:lower} an improved lower bound of $(27/2)^{1/8}>1.384$.

In order to establish these new bounds we will make use of a connection between canonical groups of face 2-coloured spherical triangulations and abelian sand-pile groups of the digraphs underlying directed Eulerian spherical embeddings. In Section \ref{sec:motivations} we will discuss the background for both of these groups as well as further motivation for addressing the above question.  

\section{Background and motivation}
\label{sec:motivations}

\subsection{Spherical latin bitrades}

Let $R$, $C$ and $S$ be (finite) sets. A \textit{partial latin square}, $P$ say, is an $|R|\times |C|$ array with rows indexed by $R$ and columns indexed by $C$ whose cells are either empty or contain an element (a \textit{symbol}) of $S$ such that each $s\in S$ occurs at most once in each row and at most once in each column. 

We can think of a partial latin square, $P$, as a subset of $R\times C\times S$; where, the triple $(r,c,s)\in P$ if and only if the cell with row $r$ and column $c$ in the array contains symbol $s$. 
Hence, we will make use of the following equivalent definition. A partial latin square is a nonempty subset $P\subset R\times C\times S$ such that if $(r_1,c_1,s_1)$ and $(r_2,c_2,s_2)$ are distinct triples of $P$, then at most one of $r_1=r_2$, $c_1=c_2$ and $s_1=s_2$ holds.

Two partial latin squares are said to be \textit{isotopic} if they are equal up to a relabelling of their sets of rows, columns and symbols. A partial latin square $P$ is said to \textit{embed in an 
abelian group $A$} if there exist injective maps $f_1:R\rightarrow A$, $f_2:C\rightarrow A$ and $f_3:S\rightarrow A$ such that $f_1(r)+f_2(c)=f_3(s)$ for all $(r,c,s)\in P$. In other words $P$ is isotopic to a partial latin square contained in the Cayley table of $A$. An abelian group $A$ is said to be a \textit{minimal abelian representation} for the partial latin square $P$ if $P$ embeds in $A$ and, for all embeddings of $P$ in $A$, the isotopic copy of $P$ in the Cayley table of $A$ generates $A$.

From here on we specify that for any partial latin square we consider, $P$ say, 
\begin{enumerate}[(i)]
\item the sets $R$, $C$ and $S$ are pairwise disjoint; and
\item for all $x\in R\cup C\cup S$ there exists a $(r,c,s)\in P$ such that $x\in\{r,c,s\}$ (that is, we exclude, from $R\cup C\cup S$, any rows, columns or symbols that do not occur in any triple of $P$).
\end{enumerate}

Let $P$ be a partial latin square with rows $R$, columns $C$ and symbols $S$.
Define $\cA_P$ to be the abelian group with generating set $R\cup C\cup S$, subject to the relations $r+c+s=0$ for each $(r,c,s)\in P$. The motivation for this definition is that if $P$ embeds in an abelian group, then it embeds in $\cA_P$ and, in particular, any minimal abelian representation $A$ of $P$ is a quotient of the finite torsion subgroup of $\cA_P$, see \cite{SimonMe} and \cite{DraKep} for details. 

Given a partial latin square $P$ with rows $R$, columns $C$ and symbols $S$, the six possible partial latin squares obtainable from $P$ by permuting the roles of $R$, $C$ and $S$ are said to be \textit{conjugate} partial latin squares.
Note that, if $P$ and $Q$ are conjugate partial latin squares, then $\cA_P\cong\cA_Q$.
The \textit{support graph} of $P$ is the graph with vertex set $R\cup C\cup S$ and an edge between vertices $x$ and $y$ if and only if there exists a $z$ such that $\{x,y,z\}=\{r,c,s\}$ for some $(r,c,s)\in P$. Observe that conjugate partial latin squares have the same support graph.

A \textit{latin bitrade} is an ordered pair $(W,B)$ of partial latin squares such that for each triple $(r_i,c_j,s_k)\in W$ (respectively $B$) there exist unique $r_{i'}\neq r_i$, $c_{j'}\neq c_j$ and $s_{k'}\neq s_k$ such that 
$$\{(r_{i'},c_j,s_k),(r_i,c_{j'},s_k),(r_i,c_j,s_{k'})\}\subseteq B\text{ (respectively W).}$$
This condition along with the definition of a partial latin square implies that $W\cap B=\emptyset$. 
The arrays in Figure \ref{fig:bitrade} correspond to a pair of partial latin squares which form a latin bitrade $(W,B)$. Note that, in this example, the two partial latin squares, $W$ and $B$, are not isotopic.
\begin{figure}
$$W:\begin{array}{c||c|c|c|c|c}
    & c_0 & c_1 & c_2 & c_3 & c_4 \\
\hline
\hline
r_0 & s_0 & s_1 & s_2 & s_3 & \\
\hline
r_1 & s_2 & s_3 & s_4 &   & s_1\\
\hline
r_2 & s_3 &   &   & s_0 & \\
\hline
r_3 &   &   & s_1 &   & s_4\\
\end{array}\qquad\qquad
B:\begin{array}{c||c|c|c|c|c}
    & c_0 & c_1 & c_2 & c_3 & c_4 \\
\hline
\hline
r_0 & s_2 & s_3 & s_1 & s_0 & \\
\hline
r_1 & s_3 & s_1 & s_2 &   & s_4\\
\hline
r_2 & s_0 &   &   & s_3 & \\
\hline
r_3 &   &   & s_4 &   & s_1\\
\end{array}
$$
\caption{A pair of partial latin squares that together form a latin bitrade.}
\label{fig:bitrade}
\end{figure}

Observe that the pair $(W,B)$ is a latin bitrade if and only if the pair $(B,W)$ is a latin bitrade. Also note that the partial latin squares forming a latin bitrade $(W,B)$ correspond to two disjoint decompositions into copies of $K_3$ of the edge set of the same vertex 3-coloured simple support graph whose vertex colour classes correspond to the sets $R$, $C$ and $S$.

Suppose that $\cG$ is a properly face 2-coloured spherical triangulation with underlying simple graph $G$, face colour classes $W$ and $B$, and a proper vertex 3-colouring with vertex colour classes $R$, $C$ and $S$.  Then the  faces of $W$ (respectively $B$) correspond to a partial latin square with rows $R$, columns $C$ and symbols $S$ (by fixing the roles of $R$, $C$ and $S$ we are arbitrarily picking one of the six conjugate partial latin squares possible). As $W$ and $B$ are decompositions of the same simple graph and, provided $|W|>1$, no face occurs in both $W$ and $B$, the pair $(W,B)$ is a latin bitrade. For example, the face 2-coloured spherical triangulation illustrated in Figure \ref{fig:a_triangulation} corresponds to the latin bitrade $(W,B)$ in Figure \ref{fig:bitrade}, the white faces corresponding to the entries in $W$ and the grey faces the entries in $B$. 

In general the partial latin squares forming a bitrade do not necessarily embed in an abelian group, see \cite{NickIan}. However, the partial latin squares forming a bitrade $(W,B)$ arising from a face 2-coloured spherical triangulation both embed in abelian groups, and hence $W$ embeds in $\cA_W$ and $B$ embeds in $\cA_B$, \cite{NickIan} and \cite{DraHamKal}; answering a question from \cite{CavenaghDrapal}. In \cite{NickIan} Cavenagh and Wanless conjectured that $\cA_W\cong\cA_B$; this was proved in a more general setting, where the underlying graphs of the triangulations are not necessarily simple, in \cite{SimonMe} as discussed in Section \ref{sec:intro}.

\subsection{Directed Eulerian spherical embeddings and\\ abelian sand-pile groups}
\label{sec:sand-pile}

Let $G$ be a graph; we will denote the degree of a vertex $v\in V(G)$ by $\deg_G(v)$ and the maximum degree over all vertices of $G$ by $\Delta(G)$. Let $\cG$ be an embedding of $G$ in the sphere. We arbitrarily fix an orientation for the vertices, and denote the rotation at a vertex $v\in V(G)$ by $\rho(v)$. Suppose $\rho(v)=(u_0,u_1,\ldots, u_{\deg_G(v)-1})$ for some $v\in V(G)$; if $\cG$ is a triangulation and $G$ is a simple graph, then $G$ contains a cycle on the set of vertices $\{u_0,u_1,\ldots, u_{\deg_G(v)-1}\}$ where, interpreting $u_{\deg_G(v)}$ as $u_0$, the edges are between $u_i$ and $u_{i+1}$. In a slight abuse of notation we will denote this cycle as $\rho(v)$.

Let $D$ be a, not necessarily simple, digraph. Label the vertices of $D$ as $v_1,v_2,\ldots,v_n$. The \textit{adjacency matrix} $A=[a_{ij}]$ of $D$ is the  $n\times n$ matrix where the entry $a_{ij}$ equals the number of arcs from vertex $v_i$ to vertex $v_j$. The \textit{asymmetric Laplacian} of $D$ is the $n\times n$ matrix $L(D)=B-A$ where $B$ is the diagonal matrix whose entry $b_{ii}$ is the out-degree of $v_i$. 

A digraph $D$ is \textit{Eulerian} if the out-degree at each vertex of $D$ equals its in-degree. In this case, for each $v\in V(D)$ we will refer to out-degree and in-degree of $v$ simply as the degree of $v$ and denote it by $\deg_D(v)$. 

Let $D$ be a connected Eulerian digraph with vertex set $V(D)=\break\{v_1,v_2,\ldots, v_n\}$; fix an $i$, where $1\leq i\leq n$. A \textit{reduced asymmetric Laplacian}, $L'(D)$, for $D$ is obtained by removing row $i$ and column $i$ from $L(D)$.  As $D$ is connected and Eulerian, the group $\mathbb{Z}^{n-1}/\mathbb{Z}^{n-1}L'(D)$ is invariant of the choice of $i$, see \cite[Lemma 4.12]{HolLevMesPerProWil}. Hence, the \textit{abelian sandpile group} of the Eulerian digraph $D$ is defined to be the group $\cS(D)=\mathbb{Z}^{n-1}/\mathbb{Z}^{n-1}L'(D)$ where $L'(D)$ is the reduced asymmetric Laplacian obtained by removing row and column $n$ of $L(D)$. (An equivalent definition of $\cS(D)$ is the finite torsion subgroup of $\mathbb{Z}^n/\mathbb{Z}^nL(D)$.)

Let $D$ be a digraph and let $v\in V(D)$. An \textit{arborescence} diverging from $v$ is a directed sub-tree of $D$ in which all the arcs are directed away from $v$. If $D$ is connected and Eulerian, and hence strongly connected, then the number of spanning arborescences diverging from a vertex $v$ does not depend on $v$, see \cite[Theorem VI.23]{Tutte}; this number is known as the \textit{tree number} of $D$ and we will denote it by $\cT(D)$. By the Matrix-Tree Theorem, \cite[Theorem VI.28]{Tutte}, $\cT(D)$ equals the determinant of $L'(D)$; which in turn equals the order of the abelian sand-pile group $\cS(D)$, see \cite[Lemma 2.8]{HolLevMesPerProWil}. A recent and comprehensive survey of results on abelian sand-pile groups of digraphs is given in \cite{HolLevMesPerProWil}.

In \cite{Mor} Rib\'o Mor uses a probabilistic argument via Suen's Inequality, \cite{Suen}, to establish an upper bound on the order of the abelian sand-pile group in an undirected planar graph in terms of the number of vertices. In the same thesis Rib\'o Mor establishes a tighter bound using non-probabilistic techniques. This bound has subsequently been improved on in \cite{BuchinSchulz}. 

Consider an embedding of a connected Eulerian digraph. If each face of the embedding is a directed cycle, equivalently the arc rotation at each vertex alternates between incoming and outgoing arcs, the embedding is called a \textit{directed Eulerian embedding}, see \cite{BonHarSir}. If the embedding is in the sphere we call it a \textit{directed Eulerian spherical embedding}. 
%
Directed Eulerian spherical embeddings are also referred to in the literature as \textit{plane alternating dimaps}. They were first studied by Tutte in \cite{Tutte-paper} and a history of their study is given in \cite{Farr}. Bonnington,  Hartsfield and \v{S}ir\'a\v{n} \cite{BonHarSir} have provided Kuratowski type theorems for directed Eulerian spherical embeddings and, in \cite{Farr}, Farr developed a theory of minors for such embeddings. Directed Eulerian embeddings in surfaces of arbitrary genus have also been studied, see \cite{BonConMorMcK} and \cite{ChenGrossHu}.

In the following Subsection we will discuss a connection between the canonical groups of face 2-coloured spherical triangulations and the abelian sand-pile groups of the underlying digraphs of directed Eulerian spherical embeddings. For a directed Eulerian spherical embedding $D$ we will denote the underlying digraph's abelian sandpile group by $\cS(D)$ and its tree number by $\cT(D)$.

\subsection{Canonical groups and abelian sand-pile groups}

Let $\cG$ be a face 2-coloured spherical triangulation with a proper vertex 3-colouring where the vertex colour classes are $R$, $C$ and $S$. Let $I\in \{R,C,S\}$; we will construct a directed Eulerian spherical embedding $D_I(\cG)$ (or simply $D_I$) with vertex set $I$. The underlying digraph will potentially have, for any pair of distinct vertices $u$ and $v$, multiple arcs from $u$ to $v$. 
Let $\{I_0,I_1,I_2\}=\{R,C,S\}$. Consider a vertex $i\in I_0$, then the rotation at $i$ is 
$\rho(i)=(u_1,v_1,u_2,v_2,\ldots, u_{\frac{1}{2}\deg_G(i)},v_{\frac{1}{2}\deg_G(i)}),$
where, without loss of generality, $u_j\in I_1$ and $v_j\in I_2$ for all $1\leq j\leq \frac{1}{2}\deg_G(i)$ and the edge $e_j$ between $u_j$ and $v_j$ in the rotation is contained in a black face. Then in $D_I$ there are $\frac{1}{2}\deg_G(i)$ outgoing arcs $a_j$ with initial vertex $i$, one for each black face, and the terminal vertex for arc $a_j$ is the vertex in $I$ contained in the white face containing edge $e_j$. Clearly, the graph $D_I$ inherits a spherical embedding from $\cG$ in which the arc rotation at each vertex alternates between incoming and outgoing arcs, so $D_I$ is Eulerian. Moreover as the sphere is connected the graph underlying $D_I$ is connected, and as $D_I$ is Eulerian it is strongly connected. Hence $D_I$ can be considered to be a directed Eulerian spherical embedding. Figure \ref{fig:with_digraph} illustrates the directed Eulerian spherical embedding $D_R$ (the arcs of which are shown as dashed) obtained from a face 2-coloured spherical triangulation.
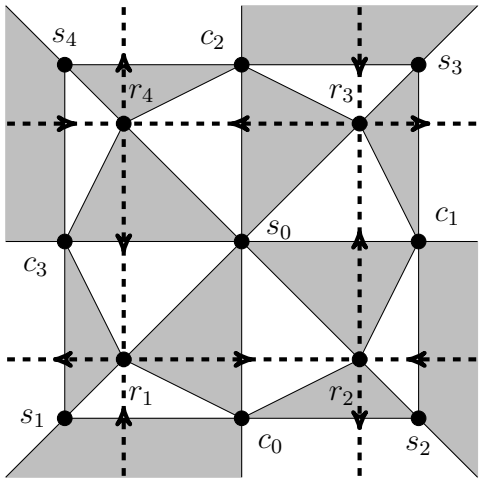
\begin{figure}[!hb]
\begin{center}
\scalebox{0.98}{\begin{tikzpicture}[fill=gray!50, scale=0.8,vertex/.style={circle,inner sep=2,fill=black,draw}]

\coordinate (v1) at (1,1);
\coordinate (v2) at (5,1);
\coordinate (v3) at (5,5);
\coordinate (v4) at (1,5);
\coordinate (v5) at (3,0);
\coordinate (v6) at (6,3);
\coordinate (v7) at (3,6);
\coordinate (v8) at (0,3);
\coordinate (v9) at (3,3);

\coordinate (v1a) at (0,0);
\coordinate (v2a) at (6,0);
\coordinate (v3a) at (6,6);
\coordinate (v4a) at (0,6);

\filldraw (v9) -- (v2) -- (v6) -- cycle;
\filldraw (v9) -- (v1) -- (v5) -- cycle;
\filldraw (v9) -- (v3) -- (v7) -- cycle;
\filldraw (v9) -- (v4) -- (v8) -- cycle;

\filldraw (v3) -- (v3a) -- (v6) -- cycle;
\filldraw (v4) -- (v4a) -- (v7) -- cycle;
\filldraw (v1) -- (v1a) -- (v8) -- cycle;
\filldraw (v2) -- (v2a) -- (v5) -- cycle;

\filldraw[color=gray!50] (-1,-1) -- (v1a) -- (v5) -- (3,-1) -- cycle;
\draw (-1,-1) -- (v1a);
\draw (v5) -- (3,-1);

\filldraw[color=gray!50] (7,-1) -- (v2a) -- (v6) -- (7,3) -- cycle;
\draw (7,-1) -- (v2a);
\draw (v6) -- (7,3);

\filldraw[color=gray!50] (7,7) -- (v3a) -- (v7) -- (3,7) -- cycle;
\draw (7,7) -- (v3a);
\draw (v7) -- (3,7);

\filldraw[color=gray!50] (-1,7) -- (v4a) -- (v8) -- (-1,3) -- cycle;
\draw (-1,7) -- (v4a);
\draw (v8) -- (-1,3);

\draw (v6) -- (v2a);
\draw (v7) -- (v3a);
\draw (v1a) -- (v5);
\draw (v8) -- (v4a);

\draw [ultra thick, dashed] (v1) -- (v2);
\draw [ultra thick, ->, >=stealth'] (3,1) -- (3.2,1);
\draw [ultra thick, dashed] (v2) -- (v3);
\draw [ultra thick, ->, >=stealth'] (5,3) -- (5,3.2);
\draw [ultra thick, dashed] (v3) -- (v4);
\draw [ultra thick, ->, >=stealth'] (3,5) -- (2.8,5);
\draw [ultra thick, dashed] (v4) -- (v1);
\draw [ultra thick, ->, >=stealth'] (1,3) -- (1,2.8);

\draw [ultra thick, dashed] (v1) -- (1,-1);
\draw [ultra thick, ->, >=stealth'] (0,1) -- (-0.2,1);
\draw [ultra thick, dashed] (v1) -- (-1,1);
\draw [ultra thick, ->, >=stealth'] (1,0) -- (1,0.2);

\draw [ultra thick, dashed] (v2) -- (7,1);
\draw [ultra thick, ->, >=stealth'] (6,1) -- (5.8,1);
\draw [ultra thick, dashed] (v2) -- (5,-1);
\draw [ultra thick, ->, >=stealth'] (5,0) -- (5,-0.2);

\draw [ultra thick, dashed] (v3) -- (7,5);
\draw [ultra thick, ->, >=stealth'] (6,5) -- (6.2,5);
\draw [ultra thick, dashed] (v3) -- (5,7);
\draw [ultra thick, ->, >=stealth'] (5,6) -- (5,5.8);

\draw [ultra thick, dashed] (v4) -- (-1,5);
\draw [ultra thick, ->, >=stealth'] (0,5) -- (0.2,5);
\draw [ultra thick, dashed] (v4) -- (1,7);
\draw [ultra thick, ->, >=stealth'] (1,6) -- (1,6.2);

\node at (v1) [vertex]{};
\node at (1.3,0.9) [label=south:$r_1$]{};
\node at (v2) [vertex]{};
\node at (4.7,0.9) [label=south:$r_2$]{};
\node at (v3) [vertex]{};
\node at (4.7,5) [label=north:$r_3$]{};
\node at (v4) [vertex]{};
\node at (1.3,5) [label=north:$r_4$]{};
\node at (v5) [vertex, label=south east:$c_0$]{};
\node at (v6) [vertex, label=north east:$c_1$]{};
\node at (v7) [vertex, label=north west:$c_2$]{};
\node at (v8) [vertex, label=south west:$c_3$]{};
\node at (v9) [vertex]{};
\node at (3.05,3.2) [label=east:$s_0$]{};

\node at (v1a) [vertex, label=west:$s_1$]{};
\node at (v2a) [vertex, label=south:$s_2$]{};
\node at (v3a) [vertex, label=east:$s_3$]{};
\node at (v4a) [vertex, label=north:$s_4$]{};

\end{tikzpicture}}
\end{center}
\caption{A face 2-coloured spherical triangulation together with corresponding directed Eulerian spherical embedding $D_R$. The vertex colour classes are $R=\{r_0,r_1,r_2,r_3,r_4\}$, where vertex $r_0$ has been placed at infinity; $C=\{c_0,c_1,c_2,c_3\}$; and $S=\{s_0,s_1,s_2,s_3,s_4\}$.}
\label{fig:with_digraph}
\end{figure}

The above construction and the following result (Lemma \ref{lem:gobackwards}) were established by Tutte \cite{Tutte-paper} (strictly speaking they were established for the dual of $\cG$); we reprove Lemma \ref{lem:gobackwards} as the construction described in the proof will be continually revisited throughout this paper.

\begin{lemma}[Tutte, \cite{Tutte-paper}]
\label{lem:gobackwards}
Given a directed Eulerian spherical embedding $D$, there exists a face 2-coloured spherical triangulation $\cG$ with a vertex 3-colouring given by the vertex sets $R$, $C$ and $S$, such that for some $I\in\{R,C,S\}$,
$D_I(\cG)\cong D.$
\end{lemma}
\begin{proof}
In short, we reverse the construction above.

Denote the faces of $D$ as $f_1, f_2, \ldots f_k$. Insert a new vertex $z_i$ into each face $f_i$ for all $1\leq i\leq k$. Consider an arc of $D$, $a$ say, that has $x$ as its initial vertex and $y$ as its terminal vertex. Then on one side of $a$ there is a new vertex $u$ and on the other a new vertex $w$. Replace $a$ with two triangular faces; a black face with vertex set $\{x,u,w\}$ and a white face with vertex set $\{y,u,w\}$. As $D$ is a directed Eulerian spherical embedding this results in a face 2-colourable spherical triangulation.
\end{proof}

We now list some observations on Lemma \ref{lem:gobackwards} and the above construction.

\begin{observation}
\label{obs:equivalences}
Let $\cG$ be a face 2-coloured spherical triangulation with underlying graph $G$ and proper vertex 3-colouring where the colour classes are $R$, $C$ and $S$. Let $I\in\{R,C,S\}$.
\begin{enumerate}[(i)]
\item If $v\in I$, then $\deg_{D_I}(v)=\frac{1}{2}\deg_G(v)$.
\item A face $f$ of size $k$ in $D_I$ corresponds to a vertex in $G$ with degree $2k$.
\item Let $\{I,J,K\}=\{R,C,S\}$. A face $f$ of size $d$ in $D_I$ corresponds to a face of size $d$ in, without loss of generality, $D_J$ and a vertex of \mbox{(out-)}degree $d$ in $D_K$. While a vertex of (out-)degree $d$ in $D_I$ corresponds to a face of size $d$ in $D_J$ and a face of size $d$ in $D_K$.
\end{enumerate}
\end{observation}

The following lemma is implicit in \cite{SimonMe}.
\begin{lemma} 
\label{lem:groups_are_iso}
Suppose that $\cG$ is a face 2-coloured spherical triangulation with a  vertex 3-colouring where the vertex colour classes are $R$, $C$ and $S$.  Then $\cS(D_R)\cong \cS(D_C)\cong \cS(D_S)\cong\cC,$ where $\cC$ is the canonical group of $\cG$.
\end{lemma}

\begin{proof}
The discussion in \cite[Section 4]{SimonMe} proves that $\cC\cong\cS(D_R)$. It is clear from the definition of $\mathcal{A}_W$ that permuting the roles of $R$, $C$ and $S$ yields vertex 3-coloured face 2-coloured triangulations of the sphere with isomorphic canonical groups. Hence $\cS(D_R)\cong \cS(D_C)\cong \cS(D_S)\cong\cC$.
\end{proof}

As mentioned earlier, the above constructions (of $D_I$, $D_J$ and $D_K$ from a face 2-coloured spherical triangulation $\cG$ and of a face 2-coloured spherical triangulation from an directed Eulerian spherical embedding) were first studied by Tutte in \cite{Tutte-paper}, where he proved that $\cT(D_I)=\cT(D_J)=\cT(D_K)$; this result is known as Tutte's Trinity Theorem, see \cite{Berman}. Note that Tutte's Trinity Theorem is a direct corollary of Lemma \ref{lem:groups_are_iso}.

In the following sections we will focus on bounding the number of spanning arborescences in the directed Eulerian spherical embedding $D_I$, where $I\in\{R,C,S\}$, obtained from a face 2-coloured spherical triangulation $\cG$ whose underlying graph $G$ is simple.

Given a directed Eulerian spherical embedding $D$ we will refer to the graph that underlies the digraph underlying $D$ simply as the graph underlying $D$.

\begin{proposition}
\label{prop:simple-implies-sufficiently-connected}
Let $\cG$ be a face 2-coloured spherical triangulation with a vertex 3-colouring with vertex sets $R$, $C$ and $S$. Suppose that the graph underlying $\cG$ is simple and contains at least four vertices. Then the graphs underlying $D_R$, $D_C$ and $D_S$ have no loops, no cut-vertices and no 2-edge-cuts.
\end{proposition}

\begin{proof}
We prove the contrapositive.  Let $G$ be the graph underlying $\cG$ and let $I\in\{R,C,S\}$.

Suppose that the graph underlying $D_I$ contains a cut-vertex, $i$ say.  Then the vertex in  $G$ that corresponds to the face of $D_I$ that contains $i$ twice in its facial walk has two edges between itself and $i$. See Figure \ref{fig:cut-vert} for an illustration of this case.

\begin{figure}[!hb]
\begin{center}
\scalebox{0.85}{
\begin{tikzpicture}[fill=gray!50, scale=0.8,vertex/.style={circle,inner sep=2,fill=black,draw}]

\coordinate (a1) at (4,0);
\coordinate (c1) at (6,1.5);
\coordinate (d1) at (6,-1.5);
\coordinate (g1) at (2,-1.5);
\coordinate (h1) at (2,1.5);

\coordinate (a2) at (14,0);
\coordinate (b2) at (15.5,0.5);
\coordinate (c2) at (16,1.5);
\coordinate (d2) at (16,-1.5);
\coordinate (e2) at (15.5,-0.5);
\coordinate (f2) at (12.5,-0.5);
\coordinate (g2) at (12,-1.5);
\coordinate (h2) at (12,1.5);
\coordinate (i2) at (12.5,0.5);


\filldraw[color=gray!50] (a2) -- (b2) -- (15.5,2.5) -- (14,2.5) -- cycle;
\draw (a2) -- (b2);
\draw (a2) -- (14,2.5);
\draw (b2) -- (15.5,2.5);

\filldraw[color=gray!50] (a2) -- (f2) -- (12.5,-2.5) -- (14,-2.5) -- cycle;
\draw (a2) -- (f2);
\draw (a2) -- (14,-2.5);
\draw (f2) -- (12.5,-2.5);

\filldraw[color=gray!50] (h2) -- (i2) -- (12.5,2.5) -- (12,2.5) -- cycle;
\draw (h2) -- (i2);
\draw (h2) -- (12,2.5);
\draw (i2) -- (12.5,2.5);

\filldraw[color=gray!50] (e2) -- (d2) -- (16,-2.5) -- (15.5,-2.5) -- cycle;
\draw (e2) -- (d2);
\draw (e2) -- (15.5,-2.5);
\draw (d2) -- (16,-2.5);

\draw [pattern=north west lines, pattern color=gray!30] (a1) to [out=143.130102354, in=323.130102354] (h1) to [out=143.130102354, in=0] (1,2) 
to [out= 190, in =90] (0,0) to [out=280, in =95] (0.5,-1.2) to [out=270, in = 200] (1.3,-2) to [out=30, in=220] (g1) [out=36.8698976458, in=216.8698976458] (a1);

\draw [pattern=north west lines, pattern color=gray!30] (a1) to [out=36.8698976458, in=216.8698976458] (c1) to [out=45, in=170] (6.8,2.2) 
to [out= 345, in =80] (8,-0.5) to [out=260, in =5] (7,-2.3) to [out=170, in =320] (d1) [out=143.130102354, in=323.130102354] (a1);

\draw [pattern=north west lines, pattern color=gray!30] (a2) to [out=143.130102354, in=323.130102354] (h2) to [out=143.130102354, in=0] (11,2) 
to [out= 190, in =90] (10,0) to [out=280, in =95] (10.5,-1.2) to [out=270, in = 200] (11.3,-2) to [out=30, in=220] (g2) [out=36.8698976458, in=216.8698976458] (a2);

\draw [pattern=north west lines, pattern color=gray!30] (a2) to [out=36.8698976458, in=216.8698976458] (c2) to [out=45, in=170] (16.8,2.2) 
to [out= 345, in =80] (18,-0.5) to [out=260, in =5] (17,-2.3) to [out=170, in =320] (d2) [out=143.130102354, in=323.130102354] (a2);


\draw [ultra thick] (a1) --  (c1);
\draw [ultra thick, ->, >=stealth'] (4.8,0.6) -- (5,0.75);
\draw [ultra thick] (a1) --  (d1);
\draw [ultra thick, ->, >=stealth'] (5.2,-0.9) -- (5,-0.75);
\draw [ultra thick] (a1) --  (g1);
\draw [ultra thick, ->, >=stealth'] (3.2,-0.6) -- (3,-0.75);
\draw [ultra thick] (a1) --  (h1);
\draw [ultra thick, ->, >=stealth'] (2.8,0.9) -- (3,0.75);

\draw (i2) -- (a2);
\draw (f2) -- (g2);
\draw (12,-2.5) -- (g2);
\draw (a2) -- (e2);
\draw (b2) -- (c2);
\draw (16,2.5) -- (c2);

\draw [ultra thick] (a2) --  (c2);
\draw [ultra thick, ->, >=stealth'] (14.8,0.6) -- (15,0.75);
\draw [ultra thick] (a2) --  (d2);
\draw [ultra thick, ->, >=stealth'] (15.2,-0.9) -- (15,-0.75);
\draw [ultra thick] (a2) --  (g2);
\draw [ultra thick, ->, >=stealth'] (13.2,-0.6) -- (13,-0.75);
\draw [ultra thick] (a2) --  (h2);
\draw [ultra thick, ->, >=stealth'] (12.8,0.9) -- (13,0.75);


\node at (a1) [vertex]{};
\node at (c1) [vertex]{};
\node at (d1) [vertex]{};
\node at (g1) [vertex]{};
\node at (h1) [vertex]{};

\node at (a2) [vertex]{};
\node at (b2) [vertex]{};
\node at (c2) [vertex]{};
\node at (d2) [vertex]{};
\node at (e2) [vertex]{};
\node at (f2) [vertex]{};
\node at (g2) [vertex]{};
\node at (h2) [vertex]{};
\node at (i2) [vertex]{};


\node at (4,-0.5) {$i$};
\node at (14.25,-0.5) {$i$};

\end{tikzpicture}
}
\end{center}

\caption{When the graph underlying $D_I$ contains a cut-vertex. The directed Eulerian spherical embedding $D_I$ is shown on the left; and the relevant faces of $\cG$ are shown on the right (the vertex constructed from the face of $D_I$ that contains vertex $i$ twice has been placed at infinity).}
\label{fig:cut-vert}
\end{figure}
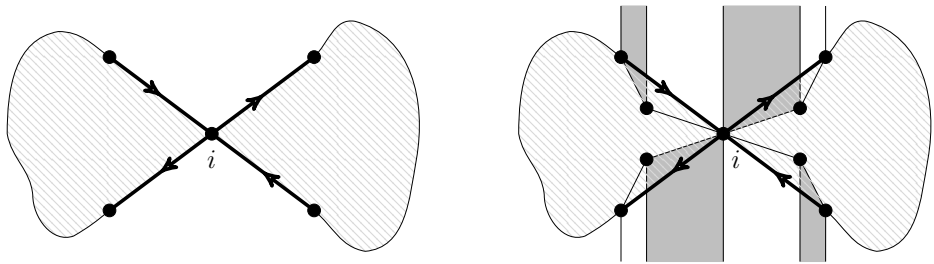


\begin{figure}[!ht]
\begin{center}
\begin{tikzpicture}[fill=gray!50, scale=0.8,vertex/.style={circle,inner sep=2,fill=black,draw}]

\coordinate (a1) at (4,0);
\coordinate (g1) at (2,-1.5);
\coordinate (h1) at (2,1.5);

\coordinate (a2) at (14,0);
\coordinate (f2) at (12.5,-0.5);
\coordinate (g2) at (12,-1.5);
\coordinate (h2) at (12,1.5);
\coordinate (i2) at (12.5,0.5);

\coordinate (x2) at (15,0);


\filldraw[color=gray!50] (a2) -- (f2) -- (12.5,-2.5) -- (14,-2.5) -- cycle;
\draw (a2) -- (f2);
\draw (a2) -- (14,-2.5);
\draw (f2) -- (12.5,-2.5);

\filldraw[color=gray!50] (h2) -- (i2) -- (12.5,2.5) -- (12,2.5) -- cycle;
\draw (h2) -- (i2);
\draw (h2) -- (12,2.5);
\draw (i2) -- (12.5,2.5);

\filldraw[color=gray!50] (14,2.5) -- (a2) -- (16.5,0) -- (16.5,2.5) -- cycle;
\draw (14,2.5) -- (a2);
\draw (a2) -- (16.5,0);

\draw [pattern=north west lines, pattern color=gray!30] (a1) to [out=143.130102354, in=323.130102354] (h1) to [out=143.130102354, in=0] (1,2) 
to [out= 190, in =90] (0,0) to [out=280, in =95] (0.5,-1.2) to [out=270, in = 200] (1.3,-2) to [out=30, in=220] (g1) [out=36.8698976458, in=216.8698976458] (a1);

\draw [pattern=north west lines, pattern color=gray!30] (a2) to [out=143.130102354, in=323.130102354] (h2) to [out=143.130102354, in=0] (11,2) 
to [out= 190, in =90] (10,0) to [out=280, in =95] (10.5,-1.2) to [out=270, in = 200] (11.3,-2) to [out=30, in=220] (g2) [out=36.8698976458, in=216.8698976458] (a2);


\draw [ultra thick] (a1) --  (g1);
\draw [ultra thick, ->, >=stealth'] (3.2,-0.6) -- (3,-0.75);
\draw [ultra thick] (a1) --  (h1);
\draw [ultra thick, ->, >=stealth'] (2.8,0.9) -- (3,0.75);

\draw [ultra thick] (5,0) circle (1cm);
\draw [ultra thick, ->, >=stealth'] (5.993,0.05) -- (6,0);

\draw [ultra thick] (15,0) circle (1cm);
\draw [ultra thick, ->, >=stealth'] (15.993,0.05) -- (16,0);

\draw (i2) -- (a2);
\draw (f2) -- (g2);
\draw (12,-2.5) -- (g2);
\draw (14,2.5) -- (a2);

\draw [ultra thick] (a2) --  (g2);
\draw [ultra thick, ->, >=stealth'] (13.2,-0.6) -- (13,-0.75);
\draw [ultra thick] (a2) --  (h2);
\draw [ultra thick, ->, >=stealth'] (12.8,0.9) -- (13,0.75);

\draw [ultra thick] (15,0) circle (1cm);


\node at (a1) [vertex]{};
\node at (g1) [vertex]{};
\node at (h1) [vertex]{};

\node at (a2) [vertex]{};
\node at (f2) [vertex]{};
\node at (g2) [vertex]{};
\node at (h2) [vertex]{};
\node at (i2) [vertex]{};
\node at (x2) [vertex]{};


\node at (4.4,0) {$x$};

\node at (5.2,0) {$f_1$};
\node at (4,2) {$f_2$};

\node at (14.4,-0.3) {$x$};

\end{tikzpicture}
\end{center}

\caption{When the graph underlying $D_I$ has no cut-vertices but does contain a loop. The directed Eulerian spherical embedding $D_I$ is shown on the left; and the relevant faces of $\cG$ are shown on the right (the vertex corresponding to $f_2$  has been placed at infinity).}
\label{fig:loop}
\end{figure}
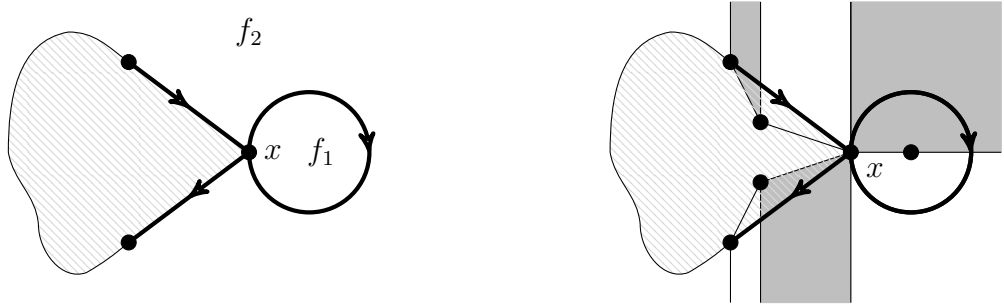

Next suppose that $D_I$ contains a loop but no cut-vertices. Let $x$ be a vertex of $D_I$ that is incident with a loop, $\ell$ say. As  $D_I$ has no cut-vertices $\ell$ is in a face $f_1$ of size one and a second face $f_2$ of size greater than or equal to one. If $f_2$ also has size one, then $G$ has three vertices, a contradiction. So $f_2$ has size at least two, and hence, in $G$, there are two edges between $x$ and the vertex corresponding to $f_2$.
See Figure \ref{fig:loop} for an illustration of this case.

So suppose that the graph underlying $D_I$ contains a 2-edge-cut. Denote the arcs that the edges of the 2-edge-cut underlie by $a_1$ and $a_2$. Then $D_I$ contains two distinct faces, $f_1$ and $f_2$ say, that both contain $a_1$ and $a_2$ in their facial walks. Thus, in $G$ there are two edges between the two vertices that correspond to $f_1$ and $f_2$.
See Figure \ref{fig:2-edge-cut} for an illustration of this case.

\begin{figure}[!ht]
\begin{center}
\begin{tikzpicture}[fill=gray!50, scale=0.8,vertex/.style={circle,inner sep=2,fill=black,draw}]

\coordinate (b1) at (5,1.5);
\coordinate (c1) at (5,-1.5);
\coordinate (d1) at (3,-1.5);
\coordinate (e1) at (3,1.5);

\coordinate (a2) at (13,0);
\coordinate (b2) at (14,1.5);
\coordinate (c2) at (14,-1.5);
\coordinate (d2) at (12,-1.5);
\coordinate (e2) at (12,1.5);


\filldraw[color=gray!50] (e2) -- (a2) -- (13,2.5) -- (12,2.5) -- cycle;
\draw (e2) -- (12,2.5);
\draw (a2) -- (13,2.5);
\draw (e2) -- (a2);

\filldraw[color=gray!50] (a2) -- (c2) -- (14,-2.5) -- (13,-2.5) -- cycle;
\draw (a2) -- (c2);
\draw (a2) -- (13,-2.5);
\draw (c2) -- (14,-2.5);

\draw (a2) -- (b2);
\draw (b2) -- (14,2.5);

\draw (a2) -- (d2);
\draw (d2) -- (12,-2.5);

\draw [pattern=north west lines, pattern color=gray!30] (2.73,0) [out=93, in=275]  to [out=80, in=280] (e1) to [out=100, in=0] (1,2) 
to [out= 190, in =90] (0,0) to [out=280, in =95] (0.5,-1.2) to [out=270, in = 200] (1.3,-2) to [out=30, in=220] (d1) [out=90, in=280] to (2.8,-0.5) [out=95, in=273] to (2.73,0);

\draw [pattern=north west lines, pattern color=gray!30] (c1) to [out=120, in=290] (b1) to [out=85, in=170] (6.3,2.2) 
to [out= 345, in =80] (7.5,-0.5) to [out=260, in =5] (6.5,-2.3) to [out=170, in =320] (c1);

\draw [pattern=north west lines, pattern color=gray!30] (11.73,0) [out=93, in=275]  to [out=80, in=280] (e2) to [out=100, in=0] (10,2) 
to [out= 190, in =90] (9,0) to [out=280, in =95] (9.5,-1.2) to [out=270, in = 200] (10.3,-2) to [out=30, in=220] (d2) [out=90, in=280] to (11.8,-0.5) [out=95, in=273] to (11.73,0);

\draw [pattern=north west lines, pattern color=gray!30] (c2) to [out=120, in=290] (b2) to [out=85, in=170] (15.3,2.2) 
to [out= 345, in =80] (16.5,-0.5) to [out=260, in =5] (15.5,-2.3) to [out=170, in =320] (c2);


\draw [ultra thick] (e1) --  (b1);
\draw [ultra thick, ->, >=stealth'] (4,1.5) -- (4.2,1.5);
\draw [ultra thick] (c1) --  (d1);
\draw [ultra thick, ->, >=stealth'] (4,-1.5) -- (3.9,-1.5);

\draw [ultra thick] (e2) --  (b2);
\draw [ultra thick, ->, >=stealth'] (13,1.5) -- (13.2,1.5);
\draw [ultra thick] (c2) --  (d2);
\draw [ultra thick, ->, >=stealth'] (13,-1.5) -- (12.8,-1.5);


\node at (b1) [vertex]{};
\node at (c1) [vertex]{};
\node at (d1) [vertex]{};
\node at (e1) [vertex]{};

\node at (a2) [vertex]{};
\node at (b2) [vertex]{};
\node at (c2) [vertex]{};
\node at (d2) [vertex]{};
\node at (e2) [vertex]{};


\node at (4,0) {$f_1$};
\node at (5.5,3) {$f_2$};

\node at (4,1.8) {$a_1$};
\node at (4,-1.8) {$a_2$};

\node at (13.3,1.8) {$a_1$};
\node at (12.7,-1.8) {$a_2$};

\end{tikzpicture}
\end{center}

\caption{When the graph underlying $D_I$ contains a 2-edge-cut. The directed Eulerian spherical embedding $D_I$ is shown on the left; and the relevant faces of $\cG$ are shown on the right (the vertex corresponding to $f_2$  has been placed at infinity).}
\label{fig:2-edge-cut}
\end{figure}
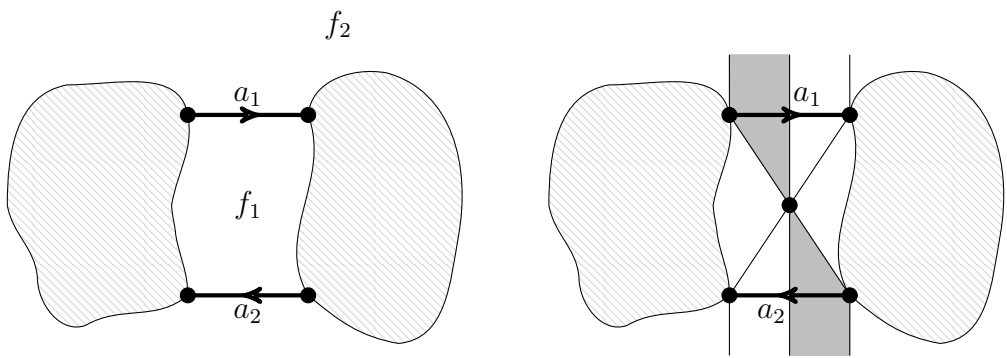

\end{proof}

\begin{proposition}
\label{prop:no-loop-no-cut-implies-simple}
Suppose that $D$ is a directed Eulerian spherical embedding with underlying graph $H$. Further suppose that $H$ has no loops, no cut-vertices and no 2-edge-cuts. Then the graph underlying the face 2-coloured spherical triangulation constructed from $D$ as in the proof of Lemma~\ref{fig:a_triangulation} is simple.
\end{proposition}

\begin{proof}
We prove the contrapositive. Let $\cG$ be a face 2-coloured spherical triangulation with a proper vertex 3-colouring with colour classes $I$, $J$ and $K$. Denote the underlying graph of $\cG$ by $G$. Suppose that $G$ has two distinct edges, say $e_1$ and $e_2$ that both have the same two end vertices; without loss of generality, let these end vertices be $i_0\in I$ and $j_0\in J$. We claim that $i_0$ is a cut-vertex or is incident with a loop in the graph underlying $D_I$, $j_0$ is a cut-vertex or is incident with a loop in the graph underlying $D_J$ and that the graph underlying $D_K$ has a 2-edge-cut.

The edges $e_1$ and $e_2$ (including their vertices) form a closed curve $L$ in the sphere. Thus, we can consider two regions (both homeomorphic to the unit disc), the interior of $L$ and the exterior. Denote the set of vertices in the interior along with the vertices $i_0$ and $j_0$ as $A$. (As $\cG$ is a triangulation $|A|\geq 3$.)

First we show that $i_0$ is either a cut-vertex or is incident with a loop in the graph underlying $D_I$ (the argument for $j_0$ and $D_J$ is similar). 

Suppose that $(A\setminus\{i_0,j_0\})\cap I$ and $I\setminus A$ are both non-empty. As $D_I$ is connected there exists a vertex $i\in A\cap I$ that has an in- or out-neighbour, in the digraph underlying $D_I$ that is not contained in $A$. As $L$ disconnects the sphere, $i$ must be in a triangular face (of $\cG$) that contains either $e_1$ or $e_2$. Thus, $i=i_0$ and is therefore a cut-vertex in the digraph underlying $D_I$. 

So, without loss of generality, suppose that $(A\setminus\{i_0,j_0\})\cap I=\emptyset$. Then there exists a $k_0\in K\cap A$ such that there is a black face $b$ with vertices $\{i_0,j_0,k_0\}$ and a white face $w$ with vertices $\{i_0,j_0,k_0\}$ such that $b$ and $w$ share the same edges between $i_0$ and $k_0$, and between $j_0$ and $k_0$. Thus $i_0$ is incident with a loop in $D_I$.



Finally, we show that the graph underlying $D_K$ has a two-edge-cut. Consider the faces in $D_K$ that correspond to the vertices $i_0$ and $j_0$ in $\cG$, denote them by $f_{i_0}$ and $f_{j_0}$ respectively.  There exist (not necessarily distinct) vertices  $k_1,k_2,k_3,k_4\in K$ such that in $\cG$:
\begin{itemize}
\item the black face containing $e_1$ has vertex set $\{i_0,j_0,k_1\}$;
\item the white face containing $e_1$ has vertex set $\{i_0,j_0,k_2\}$;
\item the black face containing $e_2$ has vertex set $\{i_0,j_0,k_3\}$; and
\item the white face containing $e_2$ has vertex set $\{i_0,j_0,k_4\}$;
\end{itemize} 
In the graph underlying $D_K$ there is an arc, $a_1$, from $k_1$ to $k_2$ and an arc, $a_2$, from $k_3$ to $k_4$; moreover both these arcs are contained in the facial walk of both $f_{i_0}$ and $f_{j_0}$. As the removal of $L$ disconnects the sphere, the removal of arcs $a_1$ and $a_2$ disconnects the graph underlying $D_K$. See Figure \ref{fig:2-edges} for an illustration of this case.

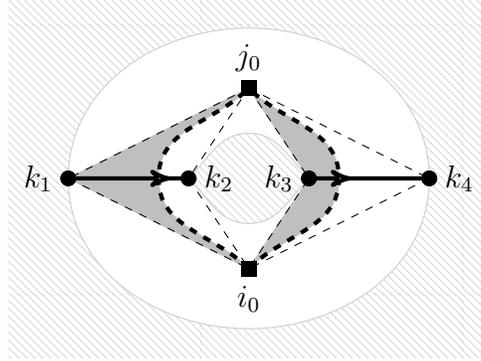
\begin{figure}[!hb]
\begin{center}
\begin{tikzpicture}[fill=gray!50, scale=0.8,vertex/.style={circle,inner sep=2,fill=black,draw},
vertex2/.style={fill=black,regular polygon, regular polygon sides=4,inner sep=2pt}
]

\coordinate (a) at (-1,0);
\coordinate (b) at (0,1.5);
\coordinate (c) at (1,0);
\coordinate (d) at (0,-1.5);
\coordinate (e) at (-3,0);
\coordinate (f) at (3,0);



%
%
%
%


\filldraw[color=white, pattern=north west lines, pattern color=gray!30] (-4,-3) -- (4,-3) -- (4,3) -- (-4,3) -- cycle;

\filldraw[color=white] (e) [out=90, in=180] to (0,2.5) [out=0, in=90] to (f) [out=270, in=0] to (0,-2.5) [out=180, in=270] to (e);   

\draw[color=gray!40] (e) [out=90, in=180] to (0,2.5) [out=0, in=90] to (f) [out=270, in=0] to (0,-2.5) [out=180, in=270] to (e);  


\filldraw[color=gray!40, pattern=north west lines, pattern color=gray!30] (a)[out=45, in=180] to (0,0.75) [out=0, in= 135] to (c) [out=225, in=0] to (0,-0.75) [out=180, in=315] to (a);

\filldraw[color=gray!50] (b) to (e) to (d) [out=135, in=270] to (-1.5,0) [out=90, in=225] to (b);
\draw[dashed] (b) -- (e);
\draw[dashed] (e) -- (d);
\draw[ultra thick,dashed] (d) [out=135, in=270] to (-1.5,0) [out=90, in=225] to (b);

\filldraw[color=gray!50] (b) to (c) to (d) [out=45, in=270] to (1.5,0) [out=90, in=315] to (b);
\draw[dashed] (b) -- (f);
\draw[dashed] (f) -- (d);
\draw[ultra thick, dashed] (d) [out=45, in=270] to (1.5,0) [out=90, in=315] to (b);

\draw[dashed] (b) -- (c) -- (d) -- (a) -- cycle;




\draw [ultra thick] (e) --  (a);
\draw [ultra thick, ->, >=stealth'] (-1.4,0) -- (-1.3,0);
\draw [ultra thick] (c) --  (f);
\draw [ultra thick, ->, >=stealth'] (1.5,0) -- (1.7,0);




\node at (a) [vertex]{};
\node at (b) [vertex2]{};
\node at (c) [vertex]{};
\node at (d) [vertex2]{};
\node at (e) [vertex]{};
\node at (f) [vertex]{};


%


\node at (0,-2) {$i_0$};
\node at (0,2) {$j_0$};

\node at (-3.5,0) {$k_1$};
\node at (-0.5,0) {$k_2$};
\node at (0.5,0) {$k_3$};
\node at (3.5,0) {$k_4$};

\end{tikzpicture}
\end{center}

\caption{The directed Eulerian spherical embedding $D_K$ when the graph underlying $\cG$ is not simple. The relevant faces of $\cG$ are superimposed (edges of $\cG$ are shown dashed with $e_1$ and $e_2$ also in bold; and relevant vertices of $\cG$ that are not in $K$ are shown as squares).} 
\label{fig:2-edges}
\end{figure}

\end{proof}

In Section \ref{sec:upper}, by considering all face 2-coloured spherical triangulations (whose underlying graphs are simple) that have a fixed number of faces in each colour class, we establish the improved upper bound. In Section \ref{sec:lower} we provide a construction for face 2-coloured spherical triangulations for which the associated directed Eulerian spherical embeddings $D_I$, where $I\in\{R,C,S\}$, have underlying digraphs with many spanning arborescences, obtaining a lower bound. Before doing so we will discuss the construction of face 2-coloured spherical triangulations that yield specific canonical groups.

\subsection{Constructing abelian groups} 

The following proposition appears in \cite{NickIan}. We provide a new proof of the result using directed Eulerian spherical embeddings.

\begin{proposition} [Cavenagh \& Wanless \cite{NickIan}]
\label{prop:cyclicgrps}
Let $m\geq 2$. There exists a face 2-coloured spherical triangulation, whose underlying graph is simple, with canonical group $\cC\cong\mathbb{Z}_{m}$.
\end{proposition}
\begin{proof}
Let $D$ be a directed Eulerian spherical embedding with two vertices, $v_0$ and $v_1$ say, and $2m$ arcs, $m$ from $v_0$ to $v_1$ and $m$  from $v_1$ to $v_0$ where the edge rotation at each vertex alternates between incoming and outgoing arcs. Then  
$L(D)=\begin{bmatrix}
m&-m\\
-m&m
\end{bmatrix}
$ and $L'(D)=[m]$, so $\cS(D)\cong\mathbb{Z}_m$.

By Lemma \ref{lem:gobackwards}, there exists a face 2-coloured spherical triangulation, with a vertex 3-colouring given by the sets $R$, $C$ and $S$ where $D=D_I$ for some $I\in\{R,C,S\}$. As $m\geq 2$, by Proposition \ref{prop:no-loop-no-cut-implies-simple}, the graph underlying the triangulation is simple. 
\end{proof}

Recursive applications of the following elementary lemma will be used to prove Proposition \ref{prop:anygroup}.

\begin{lemma}
\label{lem:gluethem}
Given two connected Eulerian digraphs $D_1$ and $D_2$ with disjoint vertex sets the graph $D$ obtained by identifying a vertex in $D_1$ with a vertex in $D_2$ has an abelian sand-pile group isomorphic to $\cS(D_1)\oplus\cS(D_2)$.
\end{lemma}
\begin{proof}
Let $v_1\in V(D)$ and $v_2\in V(D_2)$ be the vertices identified to form $D$ and denote the identified vertex as $v$. As $D_1$ and $D_2$ are connected and Eulerian, $D$ is also connected and Eulerian. Let $L'(D_1)$ (respectively $L'(D_2)$, $L'(D)$) be the reduced asymmetric Laplacian obtained by removing the row and column corresponding to $v_1$ in $L(D_1)$ (respectively $v_2$ in $D_2$ and $v$ in $D$). 
Then by applying, possibly trivial, row and column permutations to $L'(D)$ the matrix 
$$\begin{bmatrix}
L'(D_1)&\mathbf{0}\\
\mathbf{0}&L'(D_2)\\
\end{bmatrix}$$
can be obtained. Hence, $\cS(D)\cong\cS(D_1)\oplus\cS(D_2)$.
\end{proof}

\begin{proposition}
\label{prop:anygroup}
Consider an arbitrary finite abelian group $\mathbb{Z}_{m_1}\oplus\dots\oplus\mathbb{Z}_{m_k}$. Then there exists a face 2-coloured spherical triangulation with canonical group isomorphic to $\mathbb{Z}_{m_1}\oplus\dots\oplus\mathbb{Z}_{m_k}$. 
\end{proposition}
\begin{proof}
Using Proposition \ref{prop:cyclicgrps}, construct directed spherical digraphs $D_i$ for $1\leq i\leq k$ where $\cS(D_i)=\mathbb{Z}_{m_i}$. Take any spherical embedding of a tree with $k$ edges, labelled $e_1,\ldots, e_k$, and replace edge $e_i$ with the digraph underlying $D_i$, for each $1\leq i\leq k$. It is easy to see that this can be done  so that the resulting embedded digraph, $D$, is a directed Eulerian spherical embedding. The underlying digraph can also be obtained by recursive applications of Lemma \ref{lem:gluethem} and hence has an abelian sand-pile group isomorphic to $\mathbb{Z}_{m_1}\oplus\dots\oplus\mathbb{Z}_{m_k}$.
Therefore, by Lemma \ref{lem:gobackwards}, there exists a face 2-coloured triangulation that has $\mathbb{Z}_{m_1}\oplus\dots\oplus\mathbb{Z}_{m_k}$ as its canonical group.
\end{proof}

Figure \ref{fig:Z2+Z2} illustrates the construction used in the proof of Proposition \ref{prop:anygroup} in the case where the canonical group of the face 2-coloured triangulation is isomorphic to $\mathbb{Z}_2\oplus\mathbb{Z}_2$. 

\begin{figure}[!hb]
\begin{center}
\begin{tikzpicture}[fill=gray!50, scale=0.8,vertex/.style={circle,inner sep=2,fill=black,draw}]

\coordinate (r0) at (0.5,2);
\coordinate (r1) at (3.5,2);
\coordinate (r2) at (6.5,2);

\coordinate (c1) at (2,2);
\coordinate (c2) at (5,2);

\coordinate (s0) at (2,3);
\coordinate (s1) at (2,1);
\coordinate (s2) at (5,3);
\coordinate (s3) at (5,1);

\filldraw (r0) -- (c1) -- (s0) -- cycle;
\filldraw (r1) -- (c1) -- (s1) -- cycle;
\filldraw (r1) -- (c2) -- (s2) -- cycle;
\filldraw (r2) -- (c2) -- (s3) -- cycle;

\filldraw[color=gray!50] (0,0) -- (0,2) -- (r0) -- (s1) -- (2,0) -- cycle;
\draw (0,2) -- (r0);
\draw (2,0) -- (s1);
\draw (r0) -- (s1);

\filldraw[color=gray!50] (3.5,0) -- (r1) -- (s3) -- (5,0) -- cycle;
\draw (3.5,0) -- (r1);
\draw (5,0) -- (s3);
\draw (r1) -- (s3);

\filldraw[color=gray!50] (7,4) -- (7,2) -- (r2) -- (s2) -- (5,4) -- cycle;
\draw (7,2) -- (r2);
\draw (5,4) -- (s2);
\draw (r2) -- (s2);

\filldraw[color=gray!50] (3.5,4) -- (r1) -- (s0) -- (2,4) -- cycle;
\draw (3.5,4) -- (r1);
\draw (2,4) -- (s0);
\draw (r1) -- (s0);

\draw [ultra thick, dashed] (r0) to [out=25, in=180] (2,2.5) to [out=0, in =155] (r1);
\draw [ultra thick, ->, >=stealth'] (2,2.5) -- (2.2,2.5);
\draw [ultra thick, dashed] (r0) to [out=335, in=180] (2,1.5) to [out=0, in=205] (r1);
\draw [ultra thick, ->, >=stealth'] (2,1.5) -- (1.8,1.5);

\draw [ultra thick, dashed] (r0) to [out=85, in=180] (2,3.5) to [out=0, in =95] (r1);
\draw [ultra thick, ->, >=stealth'] (2,3.5) -- (1.8,3.5);
\draw [ultra thick, dashed] (r0) to [out=275, in=180] (2,0.5) to [out=0, in=265] (r1);
\draw [ultra thick, ->, >=stealth'] (2,0.5) -- (2.2,0.5);

\draw [ultra thick, dashed] (r1) to [out=25, in=180] (5,2.5) to [out=0, in =155] (r2);
\draw [ultra thick, ->, >=stealth'] (5,2.5) -- (5.2,2.5);
\draw [ultra thick, dashed] (r1) to [out=335, in=180] (5,1.5) to [out=0, in=205] (r2);
\draw [ultra thick, ->, >=stealth'] (5,1.5) -- (4.8,1.5);

\draw [ultra thick, dashed] (r1) to [out=85, in=180] (5,3.5) to [out=0, in =95] (r2);
\draw [ultra thick, ->, >=stealth'] (5,3.5) -- (4.8,3.5);
\draw [ultra thick, dashed] (r1) to [out=275, in=180] (5,0.5) to [out=0, in=265] (r2);
\draw [ultra thick, ->, >=stealth'] (5,0.5) -- (5.2,0.5);

\node at (r0) [vertex, label=north west:$r_0$]{};
\node at (r1) [vertex]{};
\node at (3.18,1.9) [label=south: $r_1$]{};
\node at (r2) [vertex, label=south east:$r_2$]{};

\node at (c1) [vertex]{};
\node at (c2) [vertex]{};
\node at (s0) [vertex]{};
\node at (s1) [vertex]{};
\node at (s2) [vertex]{};
\node at (s3) [vertex]{};

\end{tikzpicture}
\end{center}

\caption{A face 2-coloured spherical triangulation whose canonical group is isomorphic to $\mathbb{Z}_2\oplus\mathbb{Z}_2$. A vertex has been placed at infinity and the digraph $D_R$, where $R=\{r_0,r_1,r_2\}$, is shown with dashed arcs.}
\label{fig:Z2+Z2}
\end{figure}
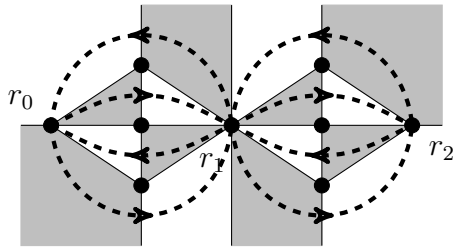

Let $A=\mathbb{Z}_{m_1}\oplus\dots\oplus\mathbb{Z}_{m_{k-1}}$ where, without loss of generality, $m_i>1$ for $1\leq i\leq k-1$. The construction in the proof of Proposition \ref{prop:anygroup} yields a triangulation $\cG$ with a proper vertex 3-colouring given by $R$, $C$ and $S$ such that $D\cong D_R(\cG)$. 
Hence a set $T$ of nonisomorphic trees on $k$ vertices yields at least $|T|/3$ nonisomorphic face 2-coloured triangulations all of which have canonical groups isomorphic to $A$. (Otter \cite{Otter} showed that the number of nonisomorphic trees on $k$ vertices is asymptotically $0.4399237 (2.95576)^k k^{−3/2}$.) 

In \cite{NickIan} Cavenagh and Wanless posed the question: which abelian groups arise as the canonical group of a properly face 2-coloured spherical triangulation whose underlying graph is simple?
Proposition \ref{prop:anygroup} does not answer this question, as, by Proposition \ref{prop:simple-implies-sufficiently-connected}, the triangulations constructed whose canonical groups are not cyclic all have underlying graphs which are not simple.

\section{Improving the upper bound}
\label{sec:upper}

In this section the underlying graphs of all the face 2-coloured spherical triangulations considered are simple. Moreover, as we are concerned with the behaviour of $m_t$ as $t\rightarrow \infty$, where $t$ is the number of faces of one colour class, in the following discussion, we take $t\geq 4$. Hence every vertex in any triangulation considered is contained in at least four faces, and no two distinct faces share the same three vertices.

Similarly to the approach taken by Rib\'o Mor in \cite{Mor}, the improved upper bound for $\lim \sup_{t\rightarrow\infty}\,(m_t)^{1/t}$ is obtained using a probabilistic argument based on Suen's Inequality. However, the results in \cite{Mor} are concerned with the growth of the number of spanning trees in terms of the number of vertices in the graph, rather than the number of arcs. A vital component of Rib\'o Mor's argument is the addition of edges to a planar graph to obtain a triangulation. The beginning of our argument follows that of \cite{Mor} (in setting up the use of a refinement of Suen's Inequality); however, as we are interested in the growth rate as the number of arcs increases, the remainder necessarily follows a different approach.

Let $\{X_i\}_{i\in\cI}$ be a finite family of Bernoulli random variables each with success probability $p_i$; i.e. $\mathbb{P}(X_i)=p_i$. A simple graph $\Gamma$ where $V(\Gamma)=\cI$ is called a \textit{dependency graph} for $\{X_i\}_{i\in\cI}$ if when two disjoint subsets of $\cI$, $A$ and $B$ say, are mutually independent, there is no edge between any vertex in $A$ and any vertex in $B$. In particular two distinct variables $X_i$ and $X_j$ are independent unless there is an edge between $i$ and $j$. For ease of notation, when discussing a dependency graph, if there exists an edge between vertices $i$ and $j$, we write $i\sim j$.
We will make use of the following refinement to Suen's Inequality  (note that in our case both Suen's Inequality and that presented in Theorem \ref{thm:janson} yield the same bound).

\begin{theorem}[Janson \cite{Janson}]
\label{thm:janson}
Let $\{X_i\}_{i\in\cI}$ be a finite family of Bernoulli random variables, where $X_i$ has success probability $p_i$, and having a dependency graph $\Gamma$. 

Let $S=\sum_{i\in\cI}X_i$; $\mu=\mathbb{E}(S)=\sum_{i\in\cI}p_i$; $\Delta =\frac{1}{2}\sum_{i\in\cI}\sum_{j\in\cI, i\sim j}\mathbb{E}(X_iX_j)$; and $\delta=\max_{i\in\cI}\sum_{k\sim i}p_k$. Then
$$\mathbb{P}(S=0)\leq \exp\left(-\mu+\Delta e^{2\delta}\right).$$
\end{theorem}

Let $\cG$ be a properly face 2-coloured triangulation with a simple underlying graph and $t$ faces of each colour such that the order of its canonical group is maximum over all such triangulations. Denote the underlying graph of $\cG$ by $G$.
Fix a vertex $i_0$ of $D_I(\cG)$ (for the remainder of this section we will write $D_I$ for $D_I(\cG)$). Let $\cR$ be a random selection of arcs selected in the following manner. For each vertex in $V(D_I)\setminus\{i_0\}=I\setminus\{i_0\}$ select, uniformly at random, one of its incoming arcs.

Then, denoting the subgraph of $D_I$ induced by the arcs of $\cR$ as $D_I[\cR]$, we have
$$\cT(D_I)=\mathbb{P}\left(D_I[\cR]\text{ is a spanning arborescence rooted at }i_0\right)\prod_{i\in I\setminus \{i_0\}}\deg_{D_I}(i).$$
Equivalently 
$$\cT(D_I)=\mathbb{P}(D_I[\cR]\text{ does not contain a directed cycle})\prod_{i\in I\setminus \{i_0\}}\deg_{D_I}(i).$$

We can now use Theorem \ref{thm:janson} to provide an upper bound for  the probability that $D_I[\cR]$ does not contain  a directed cycle (as $\cR$ contains exactly one incoming arc for each vertex not equal to $i_0$, if the underlying graph contains a cycle, it must be directed).  

Let $\cD_{I-i_0}$ denote the set of all directed cycles in $D_I$ that do not contain the vertex $i_0$. For each $\gamma\in \cD_{I-i_0}$, define:
$$X_\gamma= \left\{\begin{array}{cl}
1& \gamma\text{ is a subgraph of }D_I[\cR]\text{; and}\\
0&\text{otherwise.}
\end{array}\right.$$ 
From the definition of $\cR$ the arcs of $\gamma$ are independent events and an arc from a vertex $u$ to a vertex $v$, where there is an arc from $u$ to $v$ in $\gamma$, occurs in $\cR$ with probability $1/(\deg_{D_I}(v))$. Hence, $X_\gamma$ is a Bernoulli random variable taking the value $1$ with probability 
$$p_\gamma=\frac{1}{\prod_{v\in V(\gamma)}\deg_{D_I}(v)}.$$
Observe that $S=\sum_{\gamma\in\cD_I} X_\gamma$ counts the number of cycles in $D_I[\cR]$, and $\mathbb{P}(S=0)$ measures the probability that no cycle exists in $D_I[\cR]$.

Define a graph $\Gamma$ on the vertex set $\cD_{I-i_0}$, with an edge between the vertices in $\Gamma$ corresponding to the cycles $\alpha$ and $\beta$ of $D_I$ if and only if $\alpha$ and $\beta$ share a vertex in $D_I$. Note that two cycles in $D_I[\cR]$ can never share a vertex. Thus, $\alpha\sim \beta$ implies that $\mathbb{E}(X_\alpha X_\beta)=0$, and hence, the value of $\Delta$ from Theorem \ref{thm:janson} is zero.

Applying Theorem \ref{thm:janson}, with $\mu=\sum_{\gamma\in\cD_{I-i_0}}p_\gamma$, we have:
\begin{equation}
\label{equa:bound_for_i}
\cT(D_I)\leq \exp\left(-\mu\right)\prod_{i\in I\setminus\{i_0\}}\deg_{D_I}(i).
\end{equation}

Let $\cD_{I}$ denote the set of all directed cycles in $D_I$ and $\cD_{i_0}$ denote the set of all directed cycles in $D_I$ that contain $i_0$ (so, $\cD_{I-i_0}=\cD_I\setminus \cD_{i_0}$). Then, by Inequality (\ref{equa:bound_for_i}) we have:
{\small \begin{eqnarray*}
\cT(D_I)&\leq &  \exp\left( -\sum_{\gamma\in\cD_I\setminus \cD_{i_0}}p_\gamma \right) \prod_{i\in I\setminus\{i_0\}}\deg_{D_I}(i)\\
&&=\exp\left[
-\left(\sum_{\gamma\in\cD_{I}}p_\gamma-\sum_{\gamma\in\cD_{i_0}}p_\gamma \right)+\sum_{i\in I}\ln(\deg_{D_I}(i))
- \ln(\deg_{D_I}(i_0))\right].
\end{eqnarray*}}

Denote the set of faces in $D_I$ by $\cF_I$ and the set of faces that contain $i_0$ by $\cF_{i_0}$; and, in a slight abuse of notation, the vertex set of a face $f$ by $V(f)$ and its boundary cycle by $\gamma(f)$. As the triangulation is of a simple graph it is piecewise linear. It follows that the facial walk of any face in $\cF$ is a cycle, and so $\cF_I\subseteq\cD_I$. 

As $\cF_{i_0}\subseteq\cD_{i_0}\subseteq \cD_I$ and $\cF_{i_0}\subseteq\cF_I$,
$$\sum_{\gamma\in\cD_I} p_\gamma -\sum_{\gamma\in\cD_{i_0}} p_\gamma\geq \sum_{\gamma(f)\in\cF_I}p_{\gamma(f)} - \sum_{\gamma(f)\in\cF_{i_0}}p_{\gamma(f)}.$$
Hence,
\begin{eqnarray*}
\cT(D_I)&\leq& \exp\left[ \sum_{i\in I}\ln(\deg_{D_I}(i))-\sum_{\gamma(f)\in\cF_I}p_{\gamma(f)}\right.\\
&&\qquad\qquad\qquad\left.-\left(\ln(\deg_{D_I}(i_0))-\sum_{\gamma(f)\in\cF_{i_0}}p_{\gamma(f)} \right)
 \right].
\end{eqnarray*}

We will denote the minimum out-degree in $D_I$ by $\delta(D_I)$.

\begin{lemma}
\label{lem:new}
Suppose that $\delta(D_I)\geq 3$, or $\delta(D_I)=2$ and $|I|\geq 4$. Then 
$$\cT(D_I)<\exp \left[\sum_{i\in I}\ln(\deg_{D_I}(i))-\sum_{\gamma(f)\in\cF_I}p_{\gamma(f)}\right].$$
\end{lemma}

\begin{proof}
First, suppose that $D_I$ contains a vertex of (out-)degree greater than two. Choose $i_0$ to be such a vertex. Then $i_0$ is in $2\deg_{D_I}(i_0)$ faces and, for any of these faces, $f$ say, $p_{\gamma(f)}\leq 1/(\delta(D_I)\deg_{D_I}(i_0))$. Hence, as $\deg_{D_I}(i_0)\geq 3$,
$$\sum_{\gamma(f)\in\cF_{i_0}}p_{\gamma(f)}\leq 2\deg_{D_I}(i_0)\frac{1}{\delta(D_I)\deg_{D_I}(i_0)}\leq 1<\ln(\deg_{D_I}(i_0)).$$

So, we may assume that every vertex in $D_I$ has (out-)degree two. If there exist a vertex in four faces of size two, then $|I|=2$. If there exists a vertex in three faces of size two, then it must occur in four faces of size two, and hence $|I|=2$. If there exists a vertex for which every face containing this vertex has size greater than two, then choose such a vertex to be $i_0$, and 
$$\sum_{\gamma(f)\in\cF_{i_0}}p_{\gamma(f)}\leq 4\frac{1}{2^3}=\frac{1}{2}<\ln(2).$$

Thus, we may assume that every vertex in $D_I$ is in either exactly one or exactly two faces of size two. Assume that every vertex is in two faces of size two. Then, as $|I|\geq 4$, the other two faces have size at least four. So, 
$$\sum_{\gamma(f)\in\cF_{i_0}} p_{\gamma(f)}\leq 2\frac{1}{2^2}+2\frac{1}{2^4}<\ln(2).$$
So assume that there exists a vertex in exactly one face of size two, choose this vertex to be $i_0$. Then
$$\sum_{\gamma(f)\in\cF_{i_0}} p_{\gamma(f)}\leq \frac{1}{2^2}+3\frac{1}{2^3}<\ln(2).$$

\end{proof}

As we are interested in the growth rate of $\cT(D_I)$ as the number of arcs increases, from here on we assume our triangulation has at least eight faces in each colour class (hence, $\cT(D_I)$ has at least eight arcs). Hence, the conditions in Lemma \ref{lem:new} are met.

By Tutte's Trinity Theorem (or as a corollary of Lemma \ref{lem:groups_are_iso}), $\cT(D_R)=\cT(D_C)=\cT(D_S)$; so, by Lemma \ref{lem:new},
{\small
\begin{eqnarray*}
\cT(D_I)^3&= & \cT(D_R)\cT(D_C)\cT(D_S)\\
&<&
\exp\left[\left(\sum_{r\in R}\ln(\deg_{D_R}(r))+\sum_{c\in C}\ln(\deg_{D_C}(c))+\sum_{s\in S}\ln(\deg_{D_S}(s))\right)\right.\\
\\
&&\qquad\qquad\left.
-\left(\sum_{\gamma\in\cF_{R}}p_\gamma
+\sum_{\gamma\in\cF_{C}}p_\gamma+\sum_{\gamma\in\cF_{S}}p_\gamma\right)
\right].
\end{eqnarray*}}

Let $V=R\cup C\cup S$, then
\begin{eqnarray*}
3\ln\cT(D_I)&<&\sum_{v \in V}\ln\left(\frac{1}{2}\deg_G(v)\right)-\left(\sum_{f\in\cF_R}\frac{1}{\prod_{r\in V(f)}\deg_{D_R}(r)}\right.\\
&&\qquad\left.
 + \sum_{f\in\cF_C}\frac{1}{\prod_{c\in V(f)}\deg_{D_C}(c)} + \sum_{f\in\cF_S}\frac{1}{\prod_{s\in V(f)}\deg_{D_S}(s)}\right).
\end{eqnarray*}

Let $\{I_0,I_1,I_2\}=\{R,C,S\}$. Consider a vertex $i\in I_0$, then the rotation at $i$ is 
$$\rho(i)=(u_1,v_1,u_2,v_2,\ldots,u_{\frac{1}{2}\deg_G(i)},v_{\frac{1}{2}\deg_G(i)}),$$
where, without loss of generality, $u_j\in I_1$ and $v_j\in I_2$ for all $1\leq j\leq \frac{1}{2}\deg_G(i)$.
Note that $i$ corresponds to the face with facial walk $(u_1,u_2,\ldots,u_{\frac{1}{2}\deg_G(i)})$ in $D_{I_1}$ and the face with facial walk $(v_{\frac{1}{2}\deg_G(i)},v_{\frac{1}{2}\deg_G(i)-1},\ldots,v_2,v_1)$ in $D_{I_2}$. So, defining $\rho_1(i)=\{u_1,u_2,\ldots,u_{\frac{1}{2}\deg_G(i)}\}$ and $\rho_2(i)=\{v_1,v_2,\ldots, v_{\frac{1}{2}\deg_G(i)}\}$ we have the following upper bound for $3\ln(\cT(D_I))$.
\begin{equation}
\sum_{v\in V}\ln\left(\frac{1}{2}\deg_G(v)\right)
-\sum_{v\in V} \left(\frac{1}{\prod_{j\in \rho_1(v)}\frac{1}{2}\deg_G (j)}+ \frac{1}{\prod_{j\in \rho_2(v)}\frac{1}{2}\deg_G (j)}\right)
\label{equa:generalbound}
\end{equation}

Let $n$ denote the order of $G$ and let $n_k$ denote the number of degree $k$ vertices in $G$. Then arguing from the upper bound for $3\ln(\cT(D_I))$ given by (\ref{equa:generalbound}) we prove the following theorem.

\begin{theorem}
\label{thm:counting}
Let $m_t$ be the maximal order of the canonical group of all properly face 2-coloured spherical triangulations whose underlying graphs are simple and have $t$ faces of each colour. Then
$$\limsup_{t\rightarrow\infty}\,(m_t)^{1/t}< 6^{1/5}.$$
\end{theorem}
\begin{proof}
Let $\cG$ and $G$ be defined as in the above discussion. Define a function $g:V\rightarrow \mathbb{Z}$ by 
$$g: v\mapsto\left\{\begin{array}{cl}
2,&\text{if all the neighbours of $v$ have degree }\leq 6;\\
1,&\text{if all the neighbours of $v$ in precisely one of the two}\\ 
&\text{colour classes in $\rho(v)$ have degree }\leq 6\text{; and}\\
0,&\text{otherwise.}
\end{array}\right.$$
Let $N_4=\{v\in V(G): \deg_G(v)=4\}$ and $N_6=\{v\in V(G): \deg_G(v)=6\}$. Further let $0\leq \alpha\leq 2$ and $0\leq \beta\leq 2$ be such that $\alpha n_4=\sum_{v\in N_4}g(v)$ and $\beta n_6=\sum_{v\in N_6}g(v)$.

Rewriting the upper bound (\ref{equa:generalbound}) in terms of  the values $n_{2i}$, where $2\leq i\leq \Delta(G)/2$, and bounding the second summation in terms of $\alpha$ and $\beta$ we have
\begin{equation}3\ln(\cT(D_I))< 
\left(
\sum_{i=2}^{\Delta(G)/2}\ln(i)\,n_{2i}\right)
-\frac{\alpha n_4}{3^2}-\frac{\beta n_6}{3^3}.
\label{equa:alpha_beta_bound}
\end{equation}
As $\cG$ is a spherical triangulation, the average degree of a vertex in $G$ is $6-12/n$. So, for each vertex of degree $2i>6$ we can associate $(2i-6)/2$ degree four vertices. Hence we have that $n_4=6+\sum_{i=4}^{\Delta(G)/2}(i-3)n_{2i}$. Thus $3\ln(\cT(D_I))$ is less than
$$6\ln(2)+
\ln(3)\,n_6+
\left(
\sum_{i=4}^{\Delta(G)/2}(\ln(i)+(i-3)\ln(2))\,n_{2i}\right)
-\frac{\alpha n_4}{9}-\frac{\beta n_6}{27}.$$
Let $L$ be the set of edges of $G$ incident with a vertex of degree  
greater than six and a vertex of degree four or six. Note that $|L|\leq \sum_{i=4}^{\Delta(G)/2} (2i)n_{2i}$. Consider a vertex, $u$ say, of degree four or six. If $g(u)=0$, then there must be at least two edges in $L$ incident with $u$ and if $g(u)=1$, there is at least one edge in $L$ incident with $u$. 
Thus,
\begin{eqnarray*}
3\alpha n_4+\beta n_6&\geq& 3\left(2n_4-\min\{2n_4,|L|\}\right)+2n_6-(|L|-\min\{2n_4,|L|\})\\
&&=6n_4+2n_6-|L|-2\min\{2n_4,|L|\}.
\end{eqnarray*}
Recall that $|L|\leq \sum_{i=4}^{\Delta(G)/2} (2i)n_{2i}$ and that $2n_4=12+\sum_{i=4}^{\Delta(G)/2}2(i-3)n_{2i}$. Hence,
\begin{eqnarray*}
3\alpha n_4+\beta n_6&\geq& 6n_4+2n_6-\sum_{i=4}^{\Delta(G)/2} (2i)n_{2i}-4n_4\\
&&\qquad>2n_6+\sum_{i=4}^{\Delta(G)/2}(2(i-3)n_{2i} -(2i)n_{2i})\\
&&\qquad\qquad=2n_6-6\sum_{i=4}^{\Delta(G)/2}n_{2i}=2n_6-6(n-n_6-n_4).
\end{eqnarray*}

Therefore,
\begin{equation*}
\label{equa:casesplit}
3\ln(\cT(D_I))<6\ln(2)+\ln(3)\,n_6+
\left(
\sum_{i=4}^{\Delta(G)/2}(\ln(i)+(i-3)\ln(2))\,n_{2i}\right)
-\frac{A}{27},
\end{equation*}
where 
$$A=\left\{\begin{array}{cl}
8n_6+6n_4-6n,& \text{if } 8n_6+6n_4>6n\text{; and}\\
0,& \text{otherwise.}
\end{array}\right.$$
 
In $\sum_{i=4}^{\Delta(G)/2}(\ln(i)+(i-3)\ln(2))\,n_{2i}$ the coefficient $\ln(i)+(i-3)\ln(2)$ corresponds to the contribution of $i-2$ vertices (one of degree $2i$ and $i-3$ of degree four). Hence, the sum corresponds to the contribution of all the vertices of degree not equal to six. As $\frac{3}{2}\ln(2)\geq \frac{1}{i-2}(\ln(i)+(i-3)\ln(2))$ for all $i\geq 4$, we have that 
\begin{eqnarray}
3\ln(\cT(D_I))&<&6\ln(2)+\ln(3)n_6+\frac{3}{2}\ln(2)(n-n_6)\nonumber\\
&&=6\ln(2)+\left(\ln(3)-\frac{3}{2}\ln(2)\right)n_6+\frac{3}{2}\ln(2)n, \label{ineq:noA}
\end{eqnarray}
regardless of whether or not $A=0$.

As $A\geq 8n_6+6n_4-6n$, the following inequality is always satisfied.
{\small 
$$3\ln(\cT(D_I))<6\ln(2)+\ln(3)n_6+\sum_{i=4}^{\Delta(G)/2}(\ln(i)+(i-3)\ln(2))\,n_{2i}-\frac{8n_6+6n_4-6n}{27}.$$}
As $n_4=6+\sum_{i=4}^{\Delta(G)/2}(i-3)n_{2i}$, we have that $3\ln(\cT(D_I))$ is less than
{\small 
$$
6\left(\ln(2)-\frac{2}{9}\right)+\frac{2}{9}n+\left(\ln(3)-\frac{8}{27}\right)n_6+\sum_{i=4}^{\Delta(G)/2}\left(\ln(i)+\left(\ln(2)-\frac{2}{9}\right)(i-3)\right)n_{2i}.
$$}
Similarly to above, note that $\ln(i)+(\ln(2)-\frac{2}{9})(i-3)$ corresponds to the contribution of $i-2$ vertices (again one of degree $2i$ and $i-3$ of degree four) and the summation $\sum_{i=4}^{\Delta(G)/2}(\ln(i)+(\ln(2)-\frac{2}{9})(i-3))n_{2i}$ is the contribution of the vertices of degree not equal to six.
As $\frac{3}{2}\ln(2)-\frac{1}{9}\geq \frac{1}{i-2}\left(\ln(i)+(i-3)\left(\ln(2)-\frac{2}{9}\right)\right)$ for all $i\geq 4$ it follows that 
{
\begin{eqnarray}
3\ln(\cT(D_I))&<&6\ln(2)+\frac{2}{9}n+\left(\ln(3)-\frac{8}{27}\right)n_6+\left(\frac{3}{2}\ln(2)-\frac{1}{9}\right)(n-n_6)\nonumber
\\
&&=6\ln(2)+\left(\ln(3)-\frac{3}{2}\ln(2)-\frac{5}{27}\right)n_6+\left(\frac{3}{2}\ln(2)+\frac{1}{9}\right)n.\nonumber\\\label{ineq:withA}
\end{eqnarray}
}
As both Inequality (\ref{ineq:noA}) and Inequality (\ref{ineq:withA}) must hold,
\begin{eqnarray*}
3\ln(\cT(D_I))&<&6\ln(2)+\min\left\{\left(\ln(3)-\frac{3}{2}\ln(2)\right)n_6+\frac{3}{2}\ln(2)n,\right.\\
&&\qquad\qquad\left.\left(\ln(3)-\frac{3}{2}\ln(2)-\frac{5}{27}\right)n_6+\left(\frac{3}{2}\ln(2)+\frac{1}{9}\right)n\right\}.
\end{eqnarray*}
Fixing $n$ and letting $n_6$ vary continuously between $0$ and $n$, the maximum value of 
\begin{eqnarray*}
&\min&\left\{\left(\ln(3)-\frac{3}{2}\ln(2)\right)n_6+\frac{3}{2}\ln(2)n,\right.\\
&&\qquad\qquad\qquad\left.\left(\ln(3)-\frac{3}{2}\ln(2)-\frac{5}{27}\right)n_6+\left(\frac{3}{2}\ln(2)+\frac{1}{9}\right)n\right\}
\end{eqnarray*} 
occurs when $n_6/n=3/5$; yielding $3\ln(\cT(D_I))<\frac{3}{5}(\ln(3)+\ln(2))n=\frac{3}{5}\ln(6)n$. As $\cG$ is a triangulation of the sphere, by the Euler equation, 
$n-t=2$, where $t$ is the number of faces in one colour class;
hence,
$$\limsup_{t\rightarrow\infty}\,(m_t)^{1/t}<\limsup_{t\rightarrow\infty}\,\left(\exp\left(\frac{\ln(6)}{5}(t+2)\right)\right)^{1/t}=6^{1/5}.$$
\end{proof}

A family of face 2-coloured spherical triangulations that has attracted recent interest, see \cite{Cavenagh} and \cite{NickHomogeneous}, are triangulations that contain precisely six degree four vertices and all the other vertices have degree six, i.e. \textit{near-homogeneous} face 2-coloured spherical triangulations. Part of the motivation for their study comes from their connection to a solved case of Barnette's Conjecture \cite{Goodey}. When restricting ourselves to the near-homogeneous case we can improve the upper bound.

\begin{theorem}
\label{lem:near-homogeneous}
Let $h_t$ be the maximal order of the canonical group of all near-homogeneous properly face 2-coloured spherical triangulations whose underlying graphs are simple and have $t$ faces of each colour. Then
$$\limsup_{t\rightarrow\infty}\,(h_t)^{1/t}<\left(\exp\left(\ln(3)-\frac{2}{27}\right)\right)^{1/3}<1.4071.$$ 
\end{theorem}
\begin{proof}
As a near-homogeneous spherical triangulation has exactly six degree four vertices and every other vertex has degree six, the upper bound (\ref{equa:alpha_beta_bound}) reduces to
$ 6\ln(2)-\frac{4}{3}+\left(\ln(3)-\frac{2}{27}\right)(n-6)$,
and the result follows.
\end{proof}

\section{Improving the lower bound}
\label{sec:lower}

In \cite{GrubWan}, Grubman and Wanless analyse the effect to the order of the canonical group of face 2-coloured spherical triangulations whose underlying graphs are simple of applying several recursive constructions. They obtain a lower bound on the growth rate of $5123^{1/30}$ by using a construction that identifies a black triangle in one face 2-coloured spherical triangulation, $\cG_1$ say, with a white triangle in a second face 2-coloured spherical triangulation, $\cG_2$ say. 

Their construction can be described in terms of the related directed Eulerian spherical embeddings as follows. Let $\{I_i,J_i,K_i\}=\{R_i,C_i,S_i\}$ be the set of vertex colour classes of $\cG_i$, where $i\in\{1,2\}$. Let $a_1$ be an arc from vertex $u$ to vertex $u'$ in $D_{I_1}(\cG_1)$, and let $(u,u',x_{1,1},x_{1,2},\dots,x_{1,\ell})$ and $(u,u',y_{1,1},y_{1,2},\dots,y_{1,m})$ denote the facial walks of the two faces containing $a_1$ where the face $(u,u',x_{1,1},x_{1,2},\dots,x_{1,\ell})$ corresponds to a vertex in $J_1$ of $\cG_1$ and  the face $(u,u',y_{1,1},y_{1,2},\dots,y_{1,m})$ corresponds to a vertex in $K_1$ of $\cG_1$.
Let $a_2$ be an arc from vertex $w$ to vertex $w'$ in $D_{I_2}(\cG_2)$, and let $(w,w',x_{2,1},x_{2,2},\dots,x_{2,p})$ and $(w,w',y_{2,1},y_{2,2},\dots,y_{2,q})$ denote the facial walks of the two faces containing $a_2$ where the face $(w,w',x_{2,1},x_{2,2},\dots,x_{2,p})$ corresponds to a vertex in $J_2$ of $\cG_2$ and  the face $(w,w',y_{2,1}, y_{2,2},\dots,y_{2,q})$ corresponds to a vertex in $K_2$ of $\cG_2$.

Remove $a_1$ and the faces containing it from $D_{I_1}(\cG_1)$, and $a_2$ and the faces containing it from $D_{I_2}(\cG_2)$. Now identify $u$ and $w'$ and add an arc from $w$ to $u'$ and the faces with facial walks $$(u,x_{2,1},x_{2,2},\dots,x_{2,p},w,u',x_{1,1},x_{1,2},\dots,x_{1,\ell})$$ and $$(u,y_{2,1},y_{2,2},\dots,y_{2,q},w,u',y_{1,1},y_{1,2},\dots,y_{1,m}).$$ This yields a directed Eulerian spherical embedding $D$. Counting the spanning arborescences rooted at $u=w'$ in the underlying digraph it follows that $$\cT(D)=\cT(D_{I_1}(\cG_1))\cT(D_{I_2}(\cG_2)).$$

By considering recursive constructions applied to faces, rather than the arcs, of $D_R$, $D_C$ and $D_S$, taking care to ensure the resulting related undirected triangulations have underlying graphs that are still simple, we will provide an improved lower bound for $\limsup_{t\rightarrow \infty}(m_t)^{1/t}$. 

\begin{lemma}
\label{lem:recurse}
Let $\cG$ be a face 2-coloured spherical triangulation whose underlying graph $G$ is simple, which has a proper vertex 3-colouring given by the colour classes $R$, $C$ and $S$, and canonical group $\cC$. Suppose that $\cG$ has $t$ faces in each colour class. Further suppose that $D_I(\cG)$ for some $I\in\{R,C,S\}$ where $|I|>k>2$ contains a face, $f$ say, of size $k$ the vertices of which all have (out-)degree two.

Then there exists a face 2-coloured spherical triangulation $\cG'$, whose underlying graph is simple, with $t+2k$ faces in each colour class, a proper vertex 3-colouring given by the colour classes $R'$, $C'$ and $S'$, and with canonical group $\cC'$ such that: there exists a $I\in\{R',C',S'\}$ where $D_{I}(\cG')$ contains a face of size $k$ in which all the vertices have (out-)degree two; and
$$|\cC'|\geq \left(\sum_{j=0}^{k-1}\frac{k}{2^{j}}\binom{k-1}{j}\right)|\cC|.$$
\end{lemma}
\begin{proof}
Consider $D_I(\cG)$ and the face $f$ described in the statement of the lemma. Denote the vertices of the face $f$ by $v_0,v_1,\ldots,v_{k-1}$ so that the arcs on the boundary of the face are from $v_i$ to $v_{i+1}$, where subscripts are taken modulo $k$. Insert a new vertex into the interior of $f$, call this vertex $u$, and add an arc from $u$ to $v_j$ and an arc from $v_j$ to $u$, for all $0\leq j\leq k-1$, in such a manner as to obtain a directed Eulerian spherical embedding, $D'$ say. (We have replaced a face of size $k$ with $k$ triangular faces and $k$ digons.)

We calculate a lower bound for $\cT(D')$. Let $A$ be the set of all spanning arborescences in the digraph underlying $D_I(\cG)$ rooted at $x\not\in \{v_0,v_1,\ldots,v_{k-1}\}$.
Choose a vertex $v\in\{v_0,v_1,\ldots,v_{k-1}\}$.  Let $0\leq j\leq k-1$ and select $j$ distinct vertices from $\{v_0,v_1,\ldots,v_{k-1}\}\setminus\{v\}$, denote them $v_1', \ldots, v_j'$. For each arborescence in $A$, remove the ingoing arc with end vertex $v_i'$, for all $1\leq i\leq j$. As $\deg_{D}(v_i')=2$ this yields at least $\frac{1}{2^j}|A|$ different subgraphs. Now, to each of these subgraphs, add the arc from $v$ to $u$ and the arcs from $u$ to $v_i'$ for all $1\leq i\leq j$. This results in at least $\frac{1}{2^j}|A|$ different spanning arborescences of the digraph underlying $D'$ rooted at $x$. There were $k$ choices for $v$ and $\binom{k-1}{j}$ choices for the other $j$ vertices. Hence we have at least 
$$\left(\sum_{j=0}^{k-1}\frac{k}{2^{j}}\binom{k-1}{j}\right)|A|$$
spanning arborescences rooted at $x$ in $D'$.

To complete the proof we need to show that $D'$ corresponds to a face 2-coloured spherical triangulation, $\cG'$ say, whose underlying graph is simple, in which: there are $t+2k$ faces in each colour class; there is a vertex 3-colouring with colour classes $R'$, $C'$ and $S'$; and there exists a $I\in \{R',C',S'\}$ such that $D_I(\cG')$ has a face of size $k$ in which all the vertices have (out-)degree two.

By Lemma \ref{lem:gobackwards}, $D'$ corresponds to a face 2-coloured spherical triangulation $\cG'$ and by Proposition \ref{prop:no-loop-no-cut-implies-simple}, as $k>2$, the graph underlying $\cG'$ is simple. As $D'$ is obtained by adding $2k$ arcs (and one vertex) to $D$ it follows that $\cG'$ has $t+2k$ faces in each colour class.

The triangulation $\cG'$ can be obtained from $\cG$ by first deleting  the vertex of degree $2k$ that corresponds to $f$ in $D_I$ and all the faces and edges incident to it and replacing them with a single face of size $2k$. Denote the vertices of this new face by $w_0,w_1,\ldots,w_{2k-1}$ so that $w_{2i}=v_i$, for $0\leq i\leq k-1$, and the edges on the boundary of the face are from $w_j$ to $w_{j+1}$, where $0\leq j\leq 2k-1$ and subscripts are taken modulo $2k$. Next insert $2k+1$ new vertices, $z,z_0\ldots,z_{2k-1}$ and edges into the new face so that the rotations at the new vertices are:
$$\begin{array}{ccl}
\rho(z)&=& (z_0,z_1,\ldots,z_{2k-1}),\\
\rho(z_{2i})&=& (z,z_{2i-1},w_{2i},z_{2i+1}),\\
\rho(z_{2i+1})&=& (z,z_{2i},w_{2i},w_{2i+1},w_{2i+2},z_{2i+2}),
\end{array}$$
where $0\leq i\leq k-1$ and subscripts are taken modulo $2k$. 

So, in $\cG'$ the vertex $z$ has degree $2k$ and the vertices $z_{2i}$, where $0\leq i\leq k-1$ are all contained in precisely four faces (two white and two black). Moreover, for any proper vertex 3-colouring, each of the $z_{2i}$ belong to the same colour class, $I$ say, and in $D_I(\cG')$ there is an arc from $z_{2i}$ to $z_{2i+2}$, where subscripts are taken modulo $2k$. Hence, $D_I(\cG')$ contains a face of size $k$  in which all the vertices have (out-)degree two.
\end{proof}

Figure \ref{fig:recursion} illustrates the proof of Lemma \ref{lem:recurse} in the case where $k=4$.

\begin{figure}
\begin{center}

\begin{tabular}{ccc}

\scalebox{0.85}
{\begin{tikzpicture}[fill=gray!50, scale=0.8,vertex/.style={circle,inner sep=2,fill=black,draw}]

\coordinate (v1) at (1,1);
\coordinate (v2) at (5,1);
\coordinate (v3) at (5,5);
\coordinate (v4) at (1,5);
\coordinate (v5) at (3,0);
\coordinate (v6) at (6,3);
\coordinate (v7) at (3,6);
\coordinate (v8) at (0,3);
\coordinate (v9) at (3,3);

\coordinate (v1a) at (0,0);
\coordinate (v2a) at (6,0);
\coordinate (v3a) at (6,6);
\coordinate (v4a) at (0,6);

\filldraw (v9) -- (v2) -- (v6) -- cycle;
\filldraw (v9) -- (v1) -- (v5) -- cycle;
\filldraw (v9) -- (v3) -- (v7) -- cycle;
\filldraw (v9) -- (v4) -- (v8) -- cycle;

\filldraw (v3) -- (v3a) -- (v6) -- cycle;
\filldraw (v4) -- (v4a) -- (v7) -- cycle;
\filldraw (v1) -- (v1a) -- (v8) -- cycle;
\filldraw (v2) -- (v2a) -- (v5) -- cycle;

\filldraw[color=gray!50] (-1,-1) -- (v1a) -- (v5) -- (3,-1) -- cycle;
\draw (-1,-1) -- (v1a);
\draw (v5) -- (3,-1);

\filldraw[color=gray!50] (7,-1) -- (v2a) -- (v6) -- (7,3) -- cycle;
\draw (7,-1) -- (v2a);
\draw (v6) -- (7,3);

\filldraw[color=gray!50] (7,7) -- (v3a) -- (v7) -- (3,7) -- cycle;
\draw (7,7) -- (v3a);
\draw (v7) -- (3,7);

\filldraw[color=gray!50] (-1,7) -- (v4a) -- (v8) -- (-1,3) -- cycle;
\draw (-1,7) -- (v4a);
\draw (v8) -- (-1,3);

\draw (v6) -- (v2a);
\draw (v7) -- (v3a);
\draw (v1a) -- (v5);
\draw (v8) -- (v4a);

\draw [ultra thick, dashed] (v1) -- (v2);
\draw [ultra thick, ->, >=stealth'] (3,1) -- (3.2,1);
\draw [ultra thick, dashed] (v2) -- (v3);
\draw [ultra thick, ->, >=stealth'] (5,3) -- (5,3.2);
\draw [ultra thick, dashed] (v3) -- (v4);
\draw [ultra thick, ->, >=stealth'] (3,5) -- (2.8,5);
\draw [ultra thick, dashed] (v4) -- (v1);
\draw [ultra thick, ->, >=stealth'] (1,3) -- (1,2.8);

\draw [ultra thick, dashed] (v1) -- (1,-1);
\draw [ultra thick, ->, >=stealth'] (0,1) -- (-0.2,1);
\draw [ultra thick, dashed] (v1) -- (-1,1);
\draw [ultra thick, ->, >=stealth'] (1,0) -- (1,0.2);

\draw [ultra thick, dashed] (v2) -- (7,1);
\draw [ultra thick, ->, >=stealth'] (6,1) -- (5.8,1);
\draw [ultra thick, dashed] (v2) -- (5,-1);
\draw [ultra thick, ->, >=stealth'] (5,0) -- (5,-0.2);

\draw [ultra thick, dashed] (v3) -- (7,5);
\draw [ultra thick, ->, >=stealth'] (6,5) -- (6.2,5);
\draw [ultra thick, dashed] (v3) -- (5,7);
\draw [ultra thick, ->, >=stealth'] (5,6) -- (5,5.8);

\draw [ultra thick, dashed] (v4) -- (-1,5);
\draw [ultra thick, ->, >=stealth'] (0,5) -- (0.2,5);
\draw [ultra thick, dashed] (v4) -- (1,7);
\draw [ultra thick, ->, >=stealth'] (1,6) -- (1,6.2);

\node at (v1) [vertex]{};
\node at (1.3,0.9) [label=south:$v_0$]{};
\node at (v2) [vertex]{};
\node at (4.7,0.9) [label=south:$v_1$]{};
\node at (v3) [vertex]{};
\node at (4.7,5) [label=north:$v_2$]{};
\node at (v4) [vertex]{};
\node at (1.3,5) [label=north:$v_3$]{};
\node at (v5) [vertex]{};
\node at (v6) [vertex]{};
\node at (v7) [vertex]{};
\node at (v8) [vertex]{};
\node at (v9) [vertex]{};

\node at (v1a) [vertex]{};
\node at (v2a) [vertex]{};
\node at (v3a) [vertex]{};
\node at (v4a) [vertex]{};

\node at (3,-1) [label=south:$\cG$ with $D_I(\cG)$ shown by the dashed arcs]{};

\end{tikzpicture}}

&$\;$&

\scalebox{0.85}
{\begin{tikzpicture}[fill=gray!50, scale=0.8,vertex/.style={circle,inner sep=2,fill=black,draw}]

\coordinate (v1) at (1,1);
\coordinate (v2) at (5,1);
\coordinate (v3) at (5,5);
\coordinate (v4) at (1,5);
\coordinate (v5) at (3,0);
\coordinate (v6) at (6,3);
\coordinate (v7) at (3,6);
\coordinate (v8) at (0,3);
\coordinate (v9) at (3,3);

\coordinate (v1a) at (0,0);
\coordinate (v2a) at (6,0);
\coordinate (v3a) at (6,6);
\coordinate (v4a) at (0,6);

\coordinate (a) at (2,2);
\coordinate (b) at (4,2);
\coordinate (c) at (4,4);
\coordinate (d) at (2,4);
\coordinate (e) at (3,1.5);
\coordinate (f) at (4.5,3);
\coordinate (g) at (3,4.5);
\coordinate (h) at (1.5,3);

\filldraw (v3) -- (v3a) -- (v6) -- cycle;
\filldraw (v4) -- (v4a) -- (v7) -- cycle;
\filldraw (v1) -- (v1a) -- (v8) -- cycle;
\filldraw (v2) -- (v2a) -- (v5) -- cycle;

\filldraw (v3) -- (v3a) -- (v6) -- cycle;
\filldraw (v4) -- (v4a) -- (v7) -- cycle;
\filldraw (v1) -- (v1a) -- (v8) -- cycle;
\filldraw (v2) -- (v2a) -- (v5) -- cycle;

\filldraw (v7) -- (v3) -- (g) -- cycle;
\filldraw (v4) -- (g) -- (d) -- cycle;
\filldraw (v4) -- (v8) -- (h) -- cycle;
\filldraw (v1) -- (a) -- (h) -- cycle;
\filldraw (v1) -- (v5) -- (e) -- cycle;
\filldraw (v2) -- (b) -- (e) -- cycle;
\filldraw (v2) -- (v6) -- (f) -- cycle;
\filldraw (v3) -- (c) -- (f) -- cycle;

\filldraw (v9) -- (d) -- (h) -- cycle;
\filldraw (v9) -- (a) -- (e) -- cycle;
\filldraw (v9) -- (b) -- (f) -- cycle;
\filldraw (v9) -- (c) -- (g) -- cycle;

\filldraw[color=gray!50] (-1,-1) -- (v1a) -- (v5) -- (3,-1) -- cycle;
\draw (-1,-1) -- (v1a);
\draw (v5) -- (3,-1);

\filldraw[color=gray!50] (7,-1) -- (v2a) -- (v6) -- (7,3) -- cycle;
\draw (7,-1) -- (v2a);
\draw (v6) -- (7,3);

\filldraw[color=gray!50] (7,7) -- (v3a) -- (v7) -- (3,7) -- cycle;
\draw (7,7) -- (v3a);
\draw (v7) -- (3,7);

\filldraw[color=gray!50] (-1,7) -- (v4a) -- (v8) -- (-1,3) -- cycle;
\draw (-1,7) -- (v4a);
\draw (v8) -- (-1,3);

\draw (v6) -- (v2a);
\draw (v7) -- (v3a);
\draw (v1a) -- (v5);
\draw (v8) -- (v4a);

\draw [ultra thick, dashed] (v1) -- (v2);
\draw [ultra thick, ->, >=stealth'] (3,1) -- (3.2,1);
\draw [ultra thick, dashed] (v2) -- (v3);
\draw [ultra thick, ->, >=stealth'] (5,3) -- (5,3.2);
\draw [ultra thick, dashed] (v3) -- (v4);
\draw [ultra thick, ->, >=stealth'] (3,5) -- (2.8,5);
\draw [ultra thick, dashed] (v4) -- (v1);
\draw [ultra thick, ->, >=stealth'] (1,3) -- (1,2.8);

\draw [ultra thick, dashed] (v1) -- (1,-1);
\draw [ultra thick, ->, >=stealth'] (0,1) -- (-0.2,1);
\draw [ultra thick, dashed] (v1) -- (-1,1);
\draw [ultra thick, ->, >=stealth'] (1,0) -- (1,0.2);

\draw [ultra thick, dashed] (v2) -- (7,1);
\draw [ultra thick, ->, >=stealth'] (6,1) -- (5.8,1);
\draw [ultra thick, dashed] (v2) -- (5,-1);
\draw [ultra thick, ->, >=stealth'] (5,0) -- (5,-0.2);

\draw [ultra thick, dashed] (v3) -- (7,5);
\draw [ultra thick, ->, >=stealth'] (6,5) -- (6.2,5);
\draw [ultra thick, dashed] (v3) -- (5,7);
\draw [ultra thick, ->, >=stealth'] (5,6) -- (5,5.8);

\draw [ultra thick, dashed] (v4) -- (-1,5);
\draw [ultra thick, ->, >=stealth'] (0,5) -- (0.2,5);
\draw [ultra thick, dashed] (v4) -- (1,7);
\draw [ultra thick, ->, >=stealth'] (1,6) -- (1,6.2);

\draw [ultra thick, dashed] plot [smooth, tension=1.25] coordinates { (v1) (2.25,1.75) (v9)};
\draw [ultra thick, ->, >=stealth'] (2.3,1.8) -- (2.25,1.75);

\draw [ultra thick, dashed] plot [smooth, tension=1.25] coordinates { (v1) (1.75,2.25) (v9)};
\draw [ultra thick, ->, >=stealth'] (1.7,2.2) -- (1.75,2.25);

\draw [ultra thick, dashed] plot [smooth, tension=1.25] coordinates { (v2) (3.75,1.75) (v9)};
\draw [ultra thick, ->, >=stealth'] (3.8,1.7) -- (3.75,1.75);

\draw [ultra thick, dashed] plot [smooth, tension=1.25] coordinates { (v2) (4.25,2.25) (v9)};
\draw [ultra thick, ->, >=stealth'] (4.2,2.3) -- (4.25,2.25);

\draw [ultra thick, dashed] plot [smooth, tension=1.25] coordinates { (v3) (3.75,4.25) (v9)};
\draw [ultra thick, ->, >=stealth'] (3.7,4.2) -- (3.75,4.25);

\draw [ultra thick, dashed] plot [smooth, tension=1.25] coordinates { (v3) (4.25,3.75) (v9)};
\draw [ultra thick, ->, >=stealth'] (4.3,3.8) -- (4.25,3.75);

\draw [ultra thick, dashed] plot [smooth, tension=1.25] coordinates { (v4) (2.25,4.25) (v9)};
\draw [ultra thick, ->, >=stealth'] (2.2,4.3) -- (2.25,4.25);

\draw [ultra thick, dashed] plot [smooth, tension=1.25] coordinates { (v4) (1.75,3.75) (v9)};
\draw [ultra thick, ->, >=stealth'] (1.8,3.7) -- (1.75,3.75);

\node at (v1) [vertex]{};
\node at (1.3,0.9) [label=south:$w_0$]{};
\node at (v2) [vertex]{};
\node at (4.7,0.9) [label=south:$w_2$]{};
\node at (v3) [vertex]{};
\node at (4.7,5) [label=north:$w_4$]{};
\node at (v4) [vertex]{};
\node at (1.3,5) [label=north:$w_6$]{};
\node at (v5) [vertex, label=south east:$w_1$]{};
\node at (v6) [vertex, label=north east:$w_3$]{};
\node at (v7) [vertex, label=north west:$w_5$]{};
\node at (v8) [vertex, label=south west:$w_7$]{};
\node at (v9) [vertex]{};
\node at (3,3) [label=east:$z$]{};

\node at (2.3,2.2) [label=south west:$z_0$]{};
\node at (2.8,1.7) [label=south east:$z_1$]{};
\node at (3.8,2.3) [label=south east:$z_2$]{};
\node at (4.3,2.8) [label=north east:$z_3$]{};
\node at (3.7,3.8) [label=north east:$z_4$]{};
\node at (3.2,4.2) [label=north west:$z_5$]{};
\node at (2.2,3.8) [label=north west:$z_6$]{};
\node at (1.3,3.2) [label=south east:$z_7$]{};

\node at (a) [vertex]{};
\node at (b) [vertex]{};
\node at (c) [vertex]{};
\node at (d) [vertex]{};
\node at (e) [vertex]{};
\node at (f) [vertex]{};
\node at (g) [vertex]{};
\node at (h) [vertex]{};

\node at (v1a) [vertex]{};
\node at (v2a) [vertex]{};
\node at (v3a) [vertex]{};
\node at (v4a) [vertex]{};

\node at (3,-1) [label=south:$\cG'$ with $D'$ shown by the dashed arcs]{};

\end{tikzpicture}}

\end{tabular}

\end{center}

\caption{An illustration of an application of Lemma \ref{lem:recurse} in the case where $k=4$.} 
\label{fig:recursion}
\end{figure}
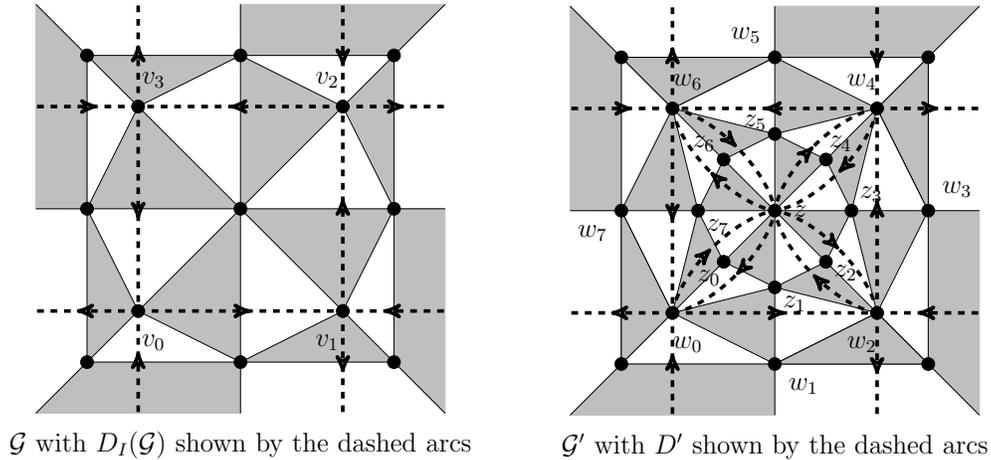

\begin{theorem}
Let $m_t$ be the maximal order of the canonical group of all properly face 2-coloured spherical triangulations whose underlying graphs are simple and have $t$ faces of each colour. Then
$$\left(\frac{27}{2}\right)^{1/8}\leq \limsup_{t\rightarrow\infty}\,(m_t)^{1/t}.$$ 
\end{theorem}

\begin{proof}
Consider the face 2-coloured spherical triangulation of a simple graph illustrated in Figure \ref{fig:with_digraph}. The vertex set $R=\{r_0,r_1,r_2,r_3,r_4\}$ forms a colour class of a vertex 3-colouring (the other classes being $\{c_0,c_1,c_2,c_3\}$ and\break $\{s_0,s_1,s_2, s_3,s_4\}$). The digraph $D_R$ contains a face of size four in which all the vertices have (out-)degree two. Repeated application of Lemma \ref{lem:recurse}, with $k=4$, obtains the result.
\end{proof}

Similar base triangulations for Lemma \ref{lem:recurse} to be recursively applied to can easily be obtained for face sizes other than four, but the resulting families have smaller growth rates.

\subsection*{Acknowledgements}
The author would like to thank anonymous referees of this paper for their helpful comments and suggestions.







\end{document}